\documentclass[10pt,reqno,oneside]{amsart}
\usepackage{graphics}
\usepackage{subfigure}
\usepackage{epsfig}
\usepackage{psfrag}
\usepackage{pstricks}
\usepackage{graphicx}
\newtheorem{theorem}{Theorem}[section]
\newtheorem{lemma}[theorem]{Lemma}

\theoremstyle{definition}

\theoremstyle{remark}
\newtheorem{remark}[theorem]{Remark}

\numberwithin{equation}{section}

\begin{document}

\title[Correctors and Field Fluctuations...]{Correctors and Field Fluctuations for the $p_{\epsilon}(x)$-Laplacian with Rough Exponents:\\ The Sublinear Growth Case}

\author[Silvia Jimenez]{Silvia Jim\'enez\\\\\tiny{Dept. of Mathematical Sciences, Worcester Polytechnic Institute\\100 Institute Road, Worcester, MA 01609-2280\\Phone: +1-508-831-5241 | Fax: +1-508-831-5824}\\\tiny{E-mail: silviajimenez@wpi.edu}}



\textit{}
\begin{abstract}
A corrector theory for the strong approximation of gradient fields inside periodic composites made from two materials 
with different power law behavior is provided.  Each material component has a distinctly different exponent appearing in the constitutive law relating gradient to flux.  The correctors are used to develop bounds on the local singularity strength for gradient fields inside micro-structured media.  The bounds are multi-scale in 
nature and can be used to measure the amplification of applied macroscopic fields by the microstructure.  The results in this paper are developed for materials having power law exponents strictly between $-1$ and zero.  
\end{abstract}


\maketitle

\section{Introduction}
In this paper, we develop a corrector theory for the strong approximation of gradient fields inside periodic composites made from two materials with different power law behavior.  Here the flux is related to the gradient $\nabla u$ by the power law $\left|\nabla u\right|^{r}\nabla u$.  Each material component has a distinctly different exponent $r$ appearing in the constitutive law relating gradient to flux.  The correctors are used to develop bounds on the local singularity strength for gradient fields inside micro-structured media.  The bounds are multi-scale in 
nature and can be used to measure the amplification of applied macroscopic fields by the microstructure.  The novelty of the work presented in this paper is that it is carried out for materials having power law exponents $r$ strictly between $-1$ and zero.  In previous work \cite{Jimenez2010}, we developed strong approximations to the gradient fields and we provided lower bounds on the $L^q$ norms ($q\geq2$) of the gradient fields inside each material that are given in terms of the correctors presented in Theorem~2.6 of \cite{Jimenez2010} for mixtures of two nonlinear power law materials with power law exponents $r$ greater than or equal to zero.  

The corrector theory for the linear case can be found in \cite{Murat1997}.  The earlier work of \cite{DalMaso1990} provides the corrector theory for homogenization of monotone operators that in our case applies to composite materials made from constituents having the same power-law growth but with rough coefficients $\sigma(x)$.  More recently, the homogenization of $p_\epsilon(x)$-Laplacian boundary value problems for smooth exponential functions $p_{\epsilon}(x)$ uniformly converging to a limit function $p_{0}(x)$ has been studied in \cite{Piatnitski2008}.  The convergence of the family of solutions for these homogenization problems is given in the topology of $L^{p_{0}(\cdot)}(\Omega)$.

Here we assume that the geometry of the composite is periodic and 
is specified by the indicator function of the sets occupied by each of the materials.  The indicator function 
of material $1$ and $2$ are denoted by $\chi_1$ and $\chi_2$, where $\chi_{1}(y) =1$ in material $1$ and is zero 
outside and $\chi_{2}(y)=1-\chi_{1}(y)$.  The constitutive law for the heterogeneous medium is described by $A:\mathbb{R}^{n}\times\mathbb{R}^{n}\rightarrow\mathbb{R}^{n}$,
\begin{equation}
	\label{A}
	\displaystyle
A\left(y,\xi\right)=\sigma(y)\left|\xi\right|^{p(y)-2}\xi,
\end{equation}	
with $\sigma(y)=\chi_{1}\left(y\right)\sigma_{1}+\chi_{2}\left(y\right)\sigma_{2}$ and $0<\sigma_{1},\sigma_{2}<\infty$; with 
$p(y)=\chi_{1}\left(y\right)p_{1}+\chi_{2}\left(y\right)p_{2}$ and $1<p_{1}\leq p_{2}<2$ or $1<p_{1}\leq2\leq p_{2}$; and with both, $\sigma(y)$ and $p(y)$, periodic in $y$, with unit period cell $Y=(0,1)^n$.  This constitutive model occurs in several mathematical models of physical processes including nonlinear dielectrics  \cite{Garroni2001,Garroni2003,Kohn1998,Talbot1994,Talbot1994-2}, fluid flow (electrorheological fluids) \cite{Antontsev2006,Ruzicka2000,Ruzicka2008}, glaciology \cite{Glowinski2003}, image restoration \cite{Levine2004}, and in the theory of deformation plasticity under longitudinal shear (anti-plane strain deformation) \cite{Atkinson1984,Suquet1993,PonteCastaneda1997,PonteCastaneda1999,Idiart2008}.  

In this paper, we study the problem of periodic homogenization associated with 
the solutions $u_{\epsilon}$ to the problems
\begin{equation}
	\label{01}
	\displaystyle
-\mbox{div}\left(A\left(\frac{x}{\epsilon},\nabla u_{\epsilon}\right)\right)=f \text{ on $\Omega$, $u_{\epsilon}\in W_{0}^{1,p_{1}}(\Omega),$}
\end{equation}
where $\Omega$ is a bounded open subset of $\mathbb{R}^{n}$, $f\in W^{-1,q_{2}}(\Omega)$, 
and $1/p_{1}+1/q_{2}=1$.  The differential operator on the left-hand side of (\ref{01}) is the $p_{\epsilon}(x)$-Laplacian.   All solutions are understood in the usual weak sense \cite{Zhikov1994}.

It was shown in Chapter~15 of \cite{Zhikov1994} that $\{u_{\epsilon}\}_{\epsilon>0}$ converges weakly in $W^{1,p_{1}}(\Omega)$ 
to the solution $u$ of the {\em homogenized} problem  
\begin{equation}
	\label{HOMOG}
	\displaystyle
-\mbox{div}\left(b\left(\nabla u\right)\right)=f \text{ on $\Omega$, $u\in W_{0}^{1,p_{1}}(\Omega)$},
\end{equation}
where the monotone map $b:\mathbb{R}^{n}\rightarrow\mathbb{R}^{n}$ (independent of $f$ and $\Omega$) can be 
obtained by solving an auxiliary problem for the operator (\ref{01}) on a periodicity cell.

The idea of homogenization is intimately related to the $\Gamma$-convergence of a suitable family of energy functionals $I_{\epsilon}$ as $\epsilon\rightarrow0$ \cite{Zhikov1994}.  Here the connection is natural in that the family of boundary value problems (\ref{01}) correspond to the Euler equations of the associated energy functionals $I_{\epsilon}$ and the solutions $u_{\epsilon}$ are their minimizers.  The homogenized solution is precisely the minimizer of the $\Gamma$-limit of the sequence $\left\{I_{\epsilon}\right\}_{\epsilon>0}$. The connections between $\Gamma$ limits and homogenization for the power-law materials studied here can be found in Chapter~15 of \cite{Zhikov1994}.  The explicit formula for the $\Gamma$-limit of the associated energy functionals for layered materials was obtained recently in \cite{Pedregal2006}.

The homogenization result found in Chapter~5 of \cite{Zhikov1994} shows that the \textit{average} of the error incurred in approximating 
$\{\nabla u_{\epsilon}\}_{\epsilon>0}$ in terms of $\nabla u$, where $u$ is the solution of (\ref{HOMOG}) 
decays to $0$.  Then again, the presence of large local fields either electric or mechanical 
often precede the onset of material failure (see, \cite{Kelly1986}).  The goal of our analysis is 
to develop tools for quantifying the effect of load transfer between length scales inside heterogeneous media.  
To this end, we present a new corrector result that approximates, inside each phase, $\nabla u_{\epsilon}$ 
up to an error that converges to zero strongly in the norm (see Section~\ref{CorrectorSection}).

The corrector result is then used to develop new tools that provide lower bounds on the local gradient field intensity inside 
micro-structured media.  The bounds are expressed in terms of the $L^q$ norms of gradients of the solutions of the 
local corrector problems.  These results provide a lower bound on the amplification of the macroscopic gradient field 
by the microstructure see, Section~\ref{SectionFluctuations}.  These bounds provide a rigorous way to assess the effect of field concentrations generated by the microgeometry without having to compute the actual solution $u_{\epsilon}$.  In \cite{Lipton2006}, similar lower bounds were established for field concentrations for mixtures of linear electrical conductors in the context of two scale convergence.  

In this paper, the corrector results are presented for layered materials (Fig.~\ref{fig:layer}) and for dispersions of inclusions embedded inside a host medium (Fig.~\ref{fig:disperse}).  
For the dispersed microstructures the included material is taken to have the lower power-law exponent than that of the host 
phase.  The reason we use dispersed and layered microstructures is that in both cases we are able to show that the homogenized solution lies in $W_{0}^{1,p_{2}}(\Omega)$, see Theorem~\ref{regularity}.  Possible extensions of this work include the study of other microstructures for which this higher order integrability condition of the homogenized solution is satisfied.  The higher order integrability is used to provide an algorithm for building correctors and construct a sequence of strong approximations 
to the gradients inside each material, see Theorem \ref{corrector}.  When the host phase has a lower power-law exponent than the included phase, one can only conclude that the homogenized solution lies in $W_{0}^{1,p_{1}}(\Omega)$ and the techniques developed here do not apply. 

The presentation of the paper is organized as follows. In Section~\ref{Problem}, we state the problem and formulate the main results.  Section~\ref{technical} contains technical lemmas and integral inequalities for the correctors used to prove the main results.  Section~\ref{main} contains the proof of the main results.  The Appendix contains all proofs of lemmas stated in Section~\ref{technical} and some remarks related to the proof of Theorem~\ref{corrector} found in Section~\ref{main}.   
\section{Statement of the Problem and Main Results}
\label{Problem}
\subsection{Notation}	
In this paper we consider two nonlinear power-law materials periodically distributed inside a domain $\Omega\subset\mathbb{R}^{n}$.  
The periodic mixture is described as follows.  We introduce the unit period cell $Y=(0,1)^{n}$ of the microstructure.  Let $F$ 
be an open subset of $Y$ of material $1$, with smooth boundary $\partial F$, such that $\overline{F}\subset Y$.  The function 
$\chi_{1}(y)=1$ inside $F$ and $0$ outside and $\chi_{2}(y)=1-\chi_{1}(y)$.  We extend $\chi_{1}(y)$ and $\chi_{2}(y)$ by 
periodicity to $\mathbb{R}^{n}$ and the $\epsilon$-periodic mixture inside $\Omega$ is described by the oscillatory characteristic 
functions $\chi_{1}^{\epsilon}(x)=\chi_{1}(x/\epsilon)$ and $\chi_{2}^{\epsilon}(x)=\chi_{2}(x/\epsilon)$.  Here we will consider 
the case where $F$ is given by a simply connected inclusion embedded inside a host material (see Fig.~\ref{fig:disperse}).  A distribution of such inclusions is 
commonly referred to as a periodic dispersion of inclusions.
\begin{figure}[h]
\centering
\includegraphics[width=0.2\textwidth]{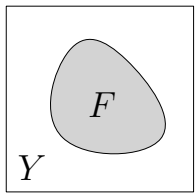}
\caption{Unit Cell: Dispersed Microstructure}
\label{fig:disperse}
\end{figure}

We also consider layered materials.  For this case the representative unit cell consists of a layer of material $1$, 
denoted by $R_{1}$, sandwiched between layers of material $2$, denoted by $R_{2}$.  The interior boundary of $R_{1}$ is denoted by 
$\Gamma$ (see Fig.~\ref{fig:layer}).  Here $\chi_{1}(y)=1$ for $y\in R_{1}$ and $0$ in $R_{2}$, and $\chi_{2}(y)=1-\chi_{1}(y)$.
\begin{figure}[h]
\centering
\includegraphics[width=0.2\textwidth]{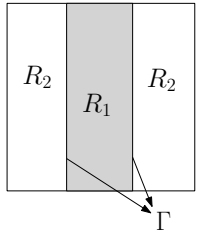}
\caption{Unit Cell: Layered Material}
\label{fig:layer}
\end{figure}

We denote by $\theta_1=\displaystyle \int_{Y}\chi_{1}(y)dy$ and $\theta_2=1-\theta_1$ the volume fractions of phase $1$ and phase $2$ inside the composite.

On the unit cell $Y$, the constitutive law for the nonlinear material is given by (\ref{A}) with exponents $p_{1}$ and $p_{2}$ 
satisfying $1<p_{1}\leq p_{2}\leq2$ or $1<p_{1}\leq2\leq p_{2}$.  Their H\"{o}lder conjugates are denoted by $q_{2}=p_{1}/(p_{1}-1)$ 
and $q_{1}=p_{2}/(p_{2}-1)$ respectively.  For $i=1,2$, $W_{per}^{1,p_{i}}(Y)$ denotes the set of all functions $u\in W^{1,p_{i}}(Y)$ 
with mean value zero that have the same trace on the opposite faces of $Y$.  Each function $u\in W_{per}^{1,p_{i}}(Y)$ can be extended 
by periodicity to a function of $W_{loc}^{1,p_{i}}(\mathbb{R}^{n})$.

The Euclidean norm and the scalar product in $\mathbb{R}^{n}$ are denoted by $\left|\cdot\right|$ and $\left(\cdot,\cdot\right)$, 
respectively.  If $A\subset\mathbb{R}^{n}$, $\left|A\right|$ denotes the Lebesgue measure and $\chi_{A}(x)$ denotes its 
characteristic function.

The constitutive law for the $\epsilon$-periodic composite is described by $A_{\epsilon}(x,\xi)=A\left(x/\epsilon,\xi\right)$, 
for every $\epsilon>0$, for every $x\in\Omega$, and for every $\xi\in\mathbb{R}^{n}$.

We have \cite{Bystrom2005} that $A$ fulfills the following conditions 
\begin{enumerate}
\item For all $\xi\in\mathbb{R}^{n}$, $A(\cdot,\xi)$ is $Y$-periodic and Lebesgue measurable.
\item $\left|A(y,0)\right|=0$ for all $y\in\mathbb{R}^{n}$.
\item Continuity: for almost every $y\in\mathbb{R}^{n}$ and for every $\xi_{i}\in\mathbb{R}^{n}$ ($i=1,2$) we have 
	\begin{equation}
	\label{ConA}
	\displaystyle \left|A(y,\xi_{1})-A(y,\xi_{2})\right|\leq C\left|\xi_{1}-\xi_{2}\right|^{\alpha(y)}\left(1+\left|\xi_{1}\right|+\left|\xi_{2}\right|\right)^{p(y)-1-\alpha(y)}
\end{equation}
where $\alpha(y)=\chi_{1}(y)\alpha_{1}(y)+\chi_{2}(y)\alpha_{2}(y)$ and
	\begin{equation*}
	\displaystyle \alpha_{i}(y)=\begin{cases}
	p_{i}-1 & \text{ if $1\leq p_{i}\leq 2$}\\
	1 & \text{ if $p_{i}\geq2$}
	\end{cases}
\end{equation*}
\item Monotonicity: for almost every $y\in\mathbb{R}^{n}$ and for every $\xi_{i}\in\mathbb{R}^{n}$ ($i=1,2$) we have
\begin{align}
\label{MonA}
	\displaystyle
	\left(A(y,\xi_{1})-A(y,\xi_{2}),\xi_{1}-\xi_{2}\right)&\geq C\left|\xi_{1}-\xi_{2}\right|^{\beta(y)}\left(\left|\xi_{1}\right|+\left|\xi_{2}\right|\right)^{p(y)-\beta(y)}\notag\\
	&\quad\geq C\left|\xi_{1}-\xi_{2}\right|^{\beta(y)}\left(1+\left|\xi_{1}\right|+\left|\xi_{2}\right|\right)^{p(y)-\beta(y)}
\end{align}
where $\beta(y)=\chi_{1}(y)\beta_{1}(y)+\chi_{2}(y)\beta_{2}(y)$ and
	\begin{equation*}
	\displaystyle \beta_{i}(y)=\begin{cases}
	2 & \text{ if $1\leq p_{i}\leq 2$}\\
	p_{i} & \text{ if $p_{i}\geq2$}
	\end{cases}
\end{equation*}	
\end{enumerate}

The structure conditions for $A$ given by (\ref{ConA}) and (\ref{MonA}) recover the ones stated in \cite{Jimenez2010} where $\alpha_{i}(y)=1$ and $\beta_{i}(y)=p_{i}$ for $i=1,2$.  In the context of this paper, the analysis when the exponents $p_{1}$ and $p_{2}$ are in the regime between $1$ and $2$ becomes more involved.  This particular set of structure conditions (or related variants) are used, for example, in \cite{DalMaso1990,Fusco1990,Braides1992,Bystrom2001}. 

\subsection{Dirichlet Boundary Value Problem}
We consider the following Dirichlet boundary value problem
\begin{equation}
	\label{Dirichlet}
	\begin{cases}
		-\mbox{div}\left(A_{\epsilon}\left(x,\nabla u_{\epsilon}\right)\right)=f \text{ on $\Omega$},\\
		u_{\epsilon}\in W_{0}^{1,p_{1}}(\Omega);  	
	\end{cases}
\end{equation}
where $f\in W^{-1,q_{2}}(\Omega)$.
 
The following homogenization result holds.
\begin{theorem}[Homogenization Theorem (see Chapter~15 of \cite{Zhikov1994})]
\label{homogenization}
As $\epsilon\rightarrow0$, the solutions $u_{\epsilon}$ of (\ref{Dirichlet}) converge weakly to $u$ 
in $W^{1,p_{1}}(\Omega)$, where $u$ is the solution of 
\begin{equation}
	\label{homogenized}
	\displaystyle -\mbox{div}\left(b\left(\nabla u\right)\right)=f \text{ on $\Omega$},
\end{equation}	
\begin{equation}
	\label{u}
	\displaystyle u \in W_{0}^{1,p_{1}}(\Omega);  	
\end{equation}
and the function $b:\mathbb{R}^{n}\rightarrow\mathbb{R}^{n}$ is defined for all $\xi\in\mathbb{R}^{n}$ by
\begin{equation}
	\label{b}
	\displaystyle b(\xi)=\int_{Y}A(y,P(y,\xi))dy,
\end{equation}
where $p:\mathbb{R}^{n}\times\mathbb{R}^{n}\rightarrow\mathbb{R}^{n}$ is defined by 
\begin{equation}
	\label{p}
   P(y,\xi)=\xi+\nabla\upsilon_{\xi}(y), 
\end{equation}
where $\upsilon_{\xi}$ is the solution to the cell problem:
\begin{equation}
	\label{cell}
	\begin{cases}
		\displaystyle
		\int_{Y}\left(A(y,\xi+\nabla\upsilon_{\xi}),\nabla w\right)dy=0 \text{, for every $w\in W_{per}^{1,p_{1}}(Y)$},\\
		\upsilon_{\xi}\in W_{per}^{1,p_{1}}(Y)	
	\end{cases}
\end{equation}	
\end{theorem}

\begin{remark}
The following a priori bound is satisfied 
\begin{equation}
	\label{aprioribound}
	\displaystyle
\sup_{\epsilon>0}\left(\int_{\Omega}\chi_{1}^{\epsilon}(x)\left|\nabla u_{\epsilon}(x)\right|^{p_{1}}dx+\int_{\Omega}\chi_{2}^{\epsilon}(x)\left|\nabla u_{\epsilon}(x)\right|^{p_{2}}dx\right)\leq C<\infty,
\end{equation}
where $C$ does not depend on $\epsilon$.  The proof of this bound is given in Lemma~\ref{proofaprioribound}.
\end{remark}

\begin{remark}
The function $b$, defined in (\ref{b}), is continuous and monotone (see Lemma~\ref{monconb}).
\end{remark}

\begin{remark}	
Since the solution $\upsilon_{\xi}$ of (\ref{cell}) can be extended by periodicity to a function of 
$W_{loc}^{1,p_{1}}(\mathbb{R}^{n})$, then (\ref{cell}) is equivalent to $-\mbox{div}(A(y,\xi+\nabla \upsilon_{\xi}(y)))=0$ 
over $\textit{D}^{'}(\mathbb{R}^{n})$, i.e., 
\begin{equation}
	\label{div with p}
	-\mbox{div}\left(A(y,P(y,\xi))\right)=0 \text{ in $\textsl{D}^{'}(\mathbb{R}^{n})$ for every $\xi\in\mathbb{R}^{n}$}.
\end{equation}

Moreover, by (\ref{cell}), we have
\begin{equation}
	\label{inner product a with p}
	\displaystyle
	\int_{Y}\left(A(y,P(y,\xi)),P(y,\xi)\right)dy=\int_{Y}\left(A(y,P(y,\xi)),\xi\right)dy=\left(b(\xi),\xi\right).
\end{equation}
\end{remark}

For $\epsilon>0$, define $P_{\epsilon}:\mathbb{R}^{n}\times\mathbb{R}^{n}\rightarrow\mathbb{R}^{n}$ by
\begin{equation}
	\label{p epsilon}
	\displaystyle P_{\epsilon}(x,\xi)=P\left(\frac{x}{\epsilon},\xi\right)=\xi+\nabla\upsilon_{\xi}\left(\frac{x}{\epsilon}\right),
\end{equation}
where $\upsilon_{\xi}$ is the unique solution of (\ref{cell}).  The functions $P$ and $P_{\epsilon}$ are easily 
seen to have the following properties
\begin{equation}
	\label{p1}
	\text{$P(\cdot,\xi)$ is $Y$-periodic and $P_{\epsilon}(x,\xi)$ is $\epsilon$-periodic in $x$.}
\end{equation}
\begin{equation}
	\label{p2}
	\displaystyle \int_{Y}P(y,\xi)dy=\xi.
\end{equation}
\begin{equation}
	\label{p3}
	P_{\epsilon}(\cdot,\xi)\rightharpoonup\xi \text{ in $L^{p_{1}}(\Omega;\mathbb{R}^{n})$ as $\epsilon\rightarrow0$.}
\end{equation}
\begin{equation}
	\label{p4}
	P(y,0)=0 \text{ for almost every $y$.}
\end{equation}
\begin{equation}
	\label{p5}
	A\left(\frac{\cdot}{\epsilon},P_{\epsilon}(\cdot,\xi)\right)\rightharpoonup b(\xi) \text{ in $L^{q_{2}}(\Omega;\mathbb{R}^{n})$, as $\epsilon\rightarrow 0$}.
\end{equation}

We now state the higher order integrability properties of the homogenized solution for periodic dispersions of inclusions 
and layered microgeometries.
\begin{theorem}
\label{regularity}
Given a periodic dispersion of inclusions or a layered material then the solution $u$ of (\ref{homogenized}) belongs to 
$W_{0}^{1,p_{2}}(\Omega)$.
\end{theorem}
\begin{remark}
\label{remarkhigher}
The proof of Theorem~\ref{regularity} \cite{Jimenez2010} uses a variational approach and considers the homogenized Lagrangian associated with $b(\xi)$ defined in (\ref{b}).  The integrability of the homogenized solution $u$ of (\ref{homogenized}) is determined by the growth of the homogenized Lagrangian with respect to its argument, which follows from the regularity of the Lagrangian.  For periodic dispersed and layered microstructures no Lavrentiev phenomenon occurs.  The proof of the regularity of the homogenized Lagrangian for periodic dispersed microstructure can be found in Chapter~14 of \cite{Zhikov1994} and for layered microstructure in \cite{Jimenez2010}.  Both proofs are valid for $p_{1},p_{2}\geq1$
\end{remark}

\subsubsection{Statement of the Corrector Theorem}
\label{CorrectorSection}
We now describe the family of correctors that provide a strong approximation of the sequence $\left\{\chi_{i}^{\epsilon}\nabla u_{\epsilon}\right\}_{\epsilon>0}$  in the $L^{p_{i}}(\Omega,\mathbb{R}^{n})$ norm, for $i=1,2$.  We denote the rescaled period cell 
with side length $\epsilon>0$  by $Y_{\epsilon}$ and write $Y_{\epsilon}^{k}=\epsilon k+Y_{\epsilon}$, where $k\in\mathbb{Z}^{n}$.  
In what follows it is convenient to define the index set $I_{\epsilon}=\left\{k\in\mathbb{Z}^{n}:Y_{\epsilon}^{k}\subset \Omega \right\}$.
For $\varphi\in L^{p_{2}}(\Omega;\mathbb{R}^{n})$, we define the local average operator $M_{\epsilon}$ associated with the 
partition $Y^k_\epsilon$, ${k\in I_{\epsilon}}$ by
\begin{equation}
	\label{approximation}
	\displaystyle M_{\epsilon}(\varphi)(x) = 
\sum_{k\in\hspace{1mm}I_{\epsilon}}{\chi_{Y_{\epsilon}^{k}}(x)\frac{1}{\left|Y_{\epsilon}^{k}\right|}\int_{Y_{\epsilon}^{k}}\varphi(y)dy};  \text{ if $\displaystyle x\in\bigcup_{k\in I_{\epsilon}}Y_{\epsilon}^{k}$,}
\end{equation}

The family of approximations of the identity map $M_{\epsilon}$ has the following properties (for a proof, see, for example \cite{Zaanen1958})
\begin{enumerate}
	\item For $i=1,2$, $\left\|M_{\epsilon}(\varphi)-\varphi\right\|_{L^{p_{i}}(\Omega;\mathbb{R}^{n})}\rightarrow0$ 
	as $\epsilon\rightarrow0$.
	\item $M_{\epsilon}(\varphi)\rightarrow\varphi$ a.e. on $\Omega$.
	\item From Jensen's inequality we have   $\left\|M_{\epsilon}(\varphi)\right\|_{L^{p_{i}}(\Omega;\mathbb{R}^{n})}\leq\left\|\varphi\right\|_{L^{p_{i}}(\Omega;\mathbb{R}^{n})}$, 
	for every $\varphi\in L^{p_{2}}(\Omega;\mathbb{R}^{n})$ and $i=1,2$.
\end{enumerate}

The strong approximation to the sequence $\left\{\chi_{i}^{\epsilon}\nabla u_{\epsilon}\right\}_{\epsilon>0}$ 
is given by the following corrector theorem. 
\begin{theorem}[Corrector Theorem]
\label{corrector}
Let $f\in W^{-1,q_{2}}(\Omega)$, let $u_{\epsilon}$ be the solutions to the problem (\ref{Dirichlet}), 
and let $u$ be the solution to problem (\ref{homogenized}).  Then, up to a subsequence, for periodic dispersions of inclusions  and 
for layered materials, we have 
\begin{equation}
 	\label{strongcvcorrector}	
 	\displaystyle
\int_{\Omega}\left|\chi_{i}^{\epsilon}(x)P_{\epsilon}\left(x,M_{\epsilon}(\nabla u)(x)\right)-\chi_{i}^{\epsilon}(x)\nabla u_{\epsilon}(x)\right|^{p_{i}}dx\rightarrow0, 
\end{equation}
as $\epsilon\rightarrow0$, for $i=1,2$.
\end{theorem}

The proof of Theorem~\ref{corrector} is given in Section~\ref{proof corrector}. 

\subsubsection{Lower Bounds on the Local Amplification of the Macroscopic Field}
\label{SectionFluctuations}
We show lower bounds on the $L^q$ norm of the gradient fields inside each material that are given in terms 
of the correctors presented in Theorem \ref{corrector}.  We begin by presenting a general lower bound that holds 
for the composition of the sequence $\{\chi_i^\epsilon\nabla u_\epsilon\}_{\epsilon>0}$ with any non-negative 
Carath\'{e}odory function. Recall that $\psi:\Omega\times\mathbb{R}^{n}\rightarrow\mathbb{R}$ is a Carath\'{e}odory 
function if $\psi(x,\cdot)$ is continuous for almost every $x\in\Omega$ and if $\psi(\cdot,\lambda)$ is measurable 
in $x$ for every $\lambda\in\mathbb{R}^n$.  The lower bound on the sequence obtained by the composition of 
$\psi(x,\cdot)$ with $\chi_i^\epsilon(x)\nabla u_\epsilon(x)$ is given by
\begin{theorem}
\label{fluctuations}
For all Carath\'{e}odory functions $\psi\geq0$ and measurable sets $D\subset\Omega$ we have 
$$\int_{D}\int_{Y}\psi\left(x,\chi_{i}(y)P\left(y,\nabla u(x)\right)\right)dydx\leq\liminf_{\epsilon\rightarrow0}\int_{D}\psi\left(x,\chi_{i}^{\epsilon}(x)\nabla u_{\epsilon}(x)\right)dx.$$
	
If the sequence $\left\{\psi\left(x,\chi_{i}^{\epsilon}(x)\nabla u_{\epsilon}(x)\right)\right\}_{\epsilon>0}$ 
is weakly convergent in $L^{1}(\Omega)$, then the inequality becomes an equality.
	
In particular, for $\psi(x,\lambda)=\left|\lambda\right|^{q}$ with $q>1$, we have
\begin{eqnarray} \int_{D}\int_{Y}\chi_{i}(y)\left|P\left(y,\nabla u(x)\right)\right|^{q}dydx\leq\liminf_{\epsilon\rightarrow0}\int_{D}\chi_{i}^{\epsilon}(x)\left|\nabla u_{\epsilon}(x)\right|^{q}dx.
\label{lq}
\end{eqnarray}
\end{theorem}

Theorem \ref{fluctuations} together with \eqref{lq} provide explicit lower bounds on the gradient field inside
each material.  It relates the local excursions of the gradient inside each phase $\chi_i^\epsilon\nabla u_\epsilon$
to the average gradient $\nabla u$ through the multiscale quantity given by the corrector $P(y,\nabla u(x))$.  It is 
clear from \eqref{lq} that the $L^q(Y\times\Omega;\mathbb{R}^{n})$ integrability of $P(y,\nabla u(x))$ provides a lower bound on the 
$L^q(\Omega;\mathbb{R}^{n})$ integrability of $\nabla u_\epsilon$.

The proof of Theorem~\ref{fluctuations} is given in Section~\ref{proof fluctuations}.
\section{Technical Lemmas}
\label{technical}
In this section we state some technical a priori bounds and convergence properties for the sequences $P_{\epsilon}$ 
defined in (\ref{p epsilon}), $\nabla u_{\epsilon}$, and $A_{\epsilon}(x, P_{\epsilon}(x,\nabla u_{\epsilon}))$ 
that are used in the proof of the main results of this paper.  The proofs of these lemmas can be found in the Appendix.   
\begin{lemma}
\label{lemma1}
For every $\xi\in\mathbb{R}^{n}$ we have
\begin{equation}
\label{Lemma 1}
	\displaystyle 
\int_{Y}\chi_{1}(y)\left|P(y,\xi)\right|^{p_{1}}dy+\int_{Y}\chi_{2}(y)\left|P(y,\xi)\right|^{p_{2}}dy\leq C\left(1+\left|\xi\right|^{p_{1}}\theta_{1}+\left|\xi\right|^{p_{2}}\theta_{2}\right)
\end{equation}
and by a change of variables, we obtain
\begin{equation}
\label{Lemma 1 epsilon}
	\displaystyle 
\int_{Y_{\epsilon}}\chi_{1}^{\epsilon}(x)\left|P_{\epsilon}(x,\xi)\right|^{p_{1}}dx+\int_{Y_{\epsilon}}\chi_{2}^{\epsilon}(x)\left|P_{\epsilon}(x,\xi)\right|^{p_{2}}dx\leq C\left(1+\left|\xi\right|^{p_{1}}\theta_{1}+\left|\xi\right|^{p_{2}}\theta_{2}\right)\left|Y_{\epsilon}\right|
	\end{equation}
\end{lemma}
\begin{lemma}
\label{lemma2}
For every $\xi_{1},\xi_{2}\in\mathbb{R}^{n}$ we have	
\begin{itemize}
	\item For $1<p_{1}\leq p_{2}\leq2$:
\begin{align}
	\label{Lemma 2 epsilon}
	\displaystyle 
&\int_{Y_{\epsilon}}\chi_{1}^{\epsilon}(x)\left|P_{\epsilon}(x,\xi_{1})-P_{\epsilon}(x,\xi_{2})\right|^{p_{1}}dx+\int_{Y_{\epsilon}}\chi_{2}^{\epsilon}(x)\left|P_{\epsilon}(x,\xi_{1})-P_{\epsilon}(x,\xi_{2})\right|^{p_{2}}dx\\
	&\leq C\left[\theta_{1}^{\frac{1}{3-p_{1}}}\left|\xi_{1}-\xi_{2}\right|^{\frac{p_{1}}{3-p_{1}}}\left(1+\left|\xi_{1}\right|^{p_{1}}\theta_{1}+\left|\xi_{2}\right|^{p_{1}}\theta_{1}+\left|\xi_{1}\right|^{p_{2}}\theta_{2}+\left|\xi_{2}\right|^{p_{2}}\theta_{2}\right)^{\frac{2-p_{1}}{3-p_{1}}}\right.\notag\\
	&\quad+ \theta_{2}^{\frac{p_{1}}{2p_{2}-p_{1}p_{2}+p_{1}}}\left|\xi_{1}-\xi_{2}\right|^{\frac{p_{1}p_{2}}{2p_{2}-p_{1}p_{2}+p_{1}}}\notag\\
	&\qquad\times	\left(1+\left|\xi_{1}\right|^{p_{1}}\theta_{1}+\left|\xi_{2}\right|^{p_{1}}\theta_{1}+\left|\xi_{1}\right|^{p_{2}}\theta_{2}+\left|\xi_{2}\right|^{p_{2}}\theta_{2}\right)^{\frac{p_{2}(2-p_{1})}{2p_{2}-p_{1}p_{2}+p_{1}}}\notag\\
	&\quad+ \theta_{1}^{\frac{p_{2}}{2p_{1}-p_{1}p_{2}+p_{2}}}\left|\xi_{1}-\xi_{2}\right|^{\frac{p_{1}p_{2}}{2p_{1}-p_{1}p_{2}+p_{2}}}\notag\\
	&\qquad\times	\left(1+\left|\xi_{1}\right|^{p_{1}}\theta_{1}+\left|\xi_{2}\right|^{p_{1}}\theta_{1}+\left|\xi_{1}\right|^{p_{2}}\theta_{2}+\left|\xi_{2}\right|^{p_{2}}\theta_{2}\right)^{\frac{p_{1}(2-p_{2})}{2p_{1}-p_{1}p_{2}+p_{2}}}\notag\\	
	&\quad\left.+ \theta_{2}^{\frac{1}{3-p_{2}}}\left|\xi_{1}-\xi_{2}\right|^{\frac{p_{2}}{3-p_{2}}}\left(1+\left|\xi_{1}\right|^{p_{1}}\theta_{1}+\left|\xi_{2}\right|^{p_{1}}\theta_{1}+\left|\xi_{1}\right|^{p_{2}}\theta_{2}+\left|\xi_{2}\right|^{p_{2}}\theta_{2}\right)^{\frac{2-p_{2}}{3-p_{2}}}\right]\left|Y_{\epsilon}\right|\notag
\end{align}
	\item For $1<p_{1}\leq 2\leq p_{2}$:
\begin{align}
	\label{Lemma 2b epsilon}
	\displaystyle 
&\int_{Y_{\epsilon}}\chi_{1}^{\epsilon}(x)\left|P_{\epsilon}(x,\xi_{1})-P_{\epsilon}(x,\xi_{2})\right|^{p_{1}}dx+\int_{Y_{\epsilon}}\chi_{2}^{\epsilon}(x)\left|P_{\epsilon}(x,\xi_{1})-P_{\epsilon}(x,\xi_{2})\right|^{p_{2}}dx\\
	&\leq C\left[\theta_{1}^{\frac{1}{3-p_{1}}}\left|\xi_{1}-\xi_{2}\right|^{\frac{p_{1}}{3-p_{1}}}\left(1+\left|\xi_{1}\right|^{p_{1}}\theta_{1}+\left|\xi_{2}\right|^{p_{1}}\theta_{1}+\left|\xi_{1}\right|^{p_{2}}\theta_{2}+\left|\xi_{2}\right|^{p_{2}}\theta_{2}\right)^{\frac{2-p_{1}}{3-p_{1}}}\right.\notag\\
	&\quad+ \theta_{2}^{\frac{p_{1}}{2p_{2}-p_{1}}}\left|\xi_{1}-\xi_{2}\right|^{\frac{p_{1}p_{2}}{2p_{2}-p_{1}}}\left(1+\left|\xi_{1}\right|^{p_{1}}\theta_{1}+\left|\xi_{2}\right|^{p_{1}}\theta_{1}+\left|\xi_{1}\right|^{p_{2}}\theta_{2}+\left|\xi_{2}\right|^{p_{2}}\theta_{2}\right)^{\frac{2(p_{2}-p_{1})}{2p_{2}-p_{1}}}\notag\\
	&\quad+ \theta_{1}\left|\xi_{1}-\xi_{2}\right|^{p_{1}}\notag\\	
	&\quad\left.+ \theta_{2}^{\frac{1}{p_{2}-1}}\left|\xi_{1}-\xi_{2}\right|^{\frac{p_{2}}{p_{2}-1}}\left(1+\left|\xi_{1}\right|^{p_{1}}\theta_{1}+\left|\xi_{2}\right|^{p_{1}}\theta_{1}+\left|\xi_{1}\right|^{p_{2}}\theta_{2}+\left|\xi_{2}\right|^{p_{2}}\theta_{2}\right)^{\frac{p_{2}-2}{p_{2}-1}}\right]\left|Y_{\epsilon}\right|\notag
\end{align}
\end{itemize}
\end{lemma}
\begin{remark}
Note the two ``extra" terms in (\ref{Lemma 2 epsilon}) of Lemma~\ref{lemma2} where there is a ``mixing" of the exponents $p_{1}$ and $p_{2}$ which do not appear in the corresponding property of $P_{\epsilon}$ given by Lemma~5.2 in \cite{Jimenez2010}.  Since Lemma~(\ref{lemma2}) is used to prove Lemma~(\ref{lemma3}), these two terms appear again in (\ref{lemma3formula}) and therefore in the proof of Theorem~\ref{corrector} and the proof of Lemma~(\ref{lemmafluctuations}) used to prove Theorem~\ref{fluctuations}.
\end{remark}

\begin{lemma}
\label{lemma3} 
Let $\varphi$ be such that $$\sup_{\epsilon>0}\left\{\int_{\Omega}\chi_{1}^{\epsilon}(x)\left|\varphi(x)\right|^{p_{1}}dx + \int_{\Omega}\chi_{2}^{\epsilon}(x)\left|\varphi(x)\right|^{p_{2}}dx\right\}<\infty, $$
and let $\Psi$ be a simple function of the form 
\begin{equation}
	\label{Psi}
	\Psi(x)=\sum_{j=0}^{m}\eta_{j}\chi_{\Omega_{j}}(x),
\end{equation}	
with $\eta_{j}\in\mathbb{R}^{n}\setminus\left\{0\right\}$, $\Omega_{j}\subset\subset\Omega$, 
$\left|\partial\Omega_{j}\right|=0$, $\Omega_{j}\cap\Omega_{k}=\emptyset$ for $j\neq k$ and $j,k=1,...,m$; 
and set $\eta_{0}=0$ and $\displaystyle \Omega_{0}=\Omega\setminus\bigcup_{j=1}^{m}\Omega_{j}$.  Then
\begin{itemize}
	\item For $1<p_{1}\leq p_{2}\leq2$:
\begin{align}
\label{lemma3formula}
	\displaystyle 	&\limsup_{\epsilon\rightarrow0}\sum_{i=1}^{2}\int_{\Omega}\chi_{i}^{\epsilon}(x)\left|P_{\epsilon}(x,M_{\epsilon}\varphi)-P_{\epsilon}(x,\Psi)\right|^{p_{i}}dx\notag\\
	& \leq C\limsup_{\epsilon\rightarrow0}\left[\left(\int_{\Omega}\chi_{1}^{\epsilon}(x)\left|\varphi-\Psi\right|^{p_{1}}dx\right)^{\frac{1}{3-p_{1}}}\left(\left|\Omega\right|+\int_{\Omega}\chi_{1}^{\epsilon}(x)\left|\varphi\right|^{p_{1}}dx\right.\right.\notag\\
	&\quad\left.+ 
\int_{\Omega}\chi_{2}^{\epsilon}(x)\left|\varphi\right|^{p_{2}}dx+ \int_{\Omega}\chi_{1}^{\epsilon}(x)\left|\Psi\right|^{p_{1}}dx + \int_{\Omega}\chi_{2}^{\epsilon}(x)\left|\Psi\right|^{p_{2}}dx\right)^{\frac{2-p_{1}}{3-p_{1}}}\\	
	&+ \left(\int_{\Omega}\chi_{2}^{\epsilon}(x)\left|\varphi-\Psi\right|^{p_{2}}dx\right)^{\frac{p_{1}}{p_{2}-p_{1}p_{2}+p_{1}}}\left(\left|\Omega\right|+\int_{\Omega}\chi_{1}^{\epsilon}(x)\left|\varphi\right|^{p_{1}}dx+ \int_{\Omega}\chi_{2}^{\epsilon}(x)\left|\varphi\right|^{p_{2}}dx\right.\notag\\ 
	&\quad\left.+ \int_{\Omega}\chi_{1}^{\epsilon}(x)\left|\Psi\right|^{p_{1}}dx + \int_{\Omega}\chi_{2}^{\epsilon}(x)\left|\Psi\right|^{p_{2}}dx\right)^{\frac{p_{2}(2-p_{1})}{p_{2}-p_{1}p_{2}+p_{1}}}\notag\\
	&+ \left(\int_{\Omega}\chi_{1}^{\epsilon}(x)\left|\varphi-\Psi\right|^{p_{1}}dx\right)^{\frac{p_{2}}{p_{1}-p_{1}p_{2}+p_{2}}}\left(\left|\Omega\right|+\int_{\Omega}\chi_{1}^{\epsilon}(x)\left|\varphi\right|^{p_{1}}dx+ 
\int_{\Omega}\chi_{2}^{\epsilon}(x)\left|\varphi\right|^{p_{2}}dx\right.\notag\\
	&\quad\left.+ \int_{\Omega}\chi_{1}^{\epsilon}(x)\left|\Psi\right|^{p_{1}}dx + \int_{\Omega}\chi_{2}^{\epsilon}(x)\left|\Psi\right|^{p_{2}}dx\right)^{\frac{p_{1}(2-p_{2})}{p_{1}-p_{1}p_{2}+p_{2}}}\notag\\
	&+ \left(\int_{\Omega}\chi_{2}^{\epsilon}(x)\left|\varphi-\Psi\right|^{p_{2}}dx\right)^{\frac{1}{3-p_{2}}}\left(\left|\Omega\right|+\int_{\Omega}\chi_{1}^{\epsilon}(x)\left|\varphi\right|^{p_{1}}dx+ 
\int_{\Omega}\chi_{2}^{\epsilon}(x)\left|\varphi\right|^{p_{2}}dx\right.\notag\\
	&\quad\left.\left.+ \int_{\Omega}\chi_{1}^{\epsilon}(x)\left|\Psi\right|^{p_{1}}dx + \int_{\Omega}\chi_{2}^{\epsilon}(x)\left|\Psi\right|^{p_{2}}dx\right)^{\frac{2-p_{2}}{3-p_{2}}}\right]\notag
\end{align}
	\item For $1<p_{1}\leq2\leq p_{2}$:
\begin{align}
\label{lemma3formula2}
	\displaystyle 	&\limsup_{\epsilon\rightarrow0}\sum_{i=1}^{2}\int_{\Omega}\chi_{i}^{\epsilon}(x)\left|P_{\epsilon}(x,M_{\epsilon}\varphi)-P_{\epsilon}(x,\Psi)\right|^{p_{i}}dx\notag\\
	& \leq C\limsup_{\epsilon\rightarrow0}\left[\left(\int_{\Omega}\chi_{1}^{\epsilon}(x)\left|\varphi-\Psi\right|^{p_{1}}dx\right)^{\frac{1}{3-p_{1}}}\left(\left|\Omega\right|+\int_{\Omega}\chi_{1}^{\epsilon}(x)\left|\varphi\right|^{p_{1}}dx\right.\right.\notag\\
	&\quad\left.+ 
\int_{\Omega}\chi_{2}^{\epsilon}(x)\left|\varphi\right|^{p_{2}}dx + \int_{\Omega}\chi_{1}^{\epsilon}(x)\left|\Psi\right|^{p_{1}}dx + \int_{\Omega}\chi_{2}^{\epsilon}(x)\left|\Psi\right|^{p_{2}}dx\right)^{\frac{2-p_{1}}{3-p_{1}}}\\
	&
+\left(\int_{\Omega}\chi_{2}^{\epsilon}(x)\left|\varphi-\Psi\right|^{p_{2}}dx\right)^{\frac{p_{1}}{2p_{2}-p_{1}}}\left(\left|\Omega\right|+\int_{\Omega}\chi_{1}^{\epsilon}(x)\left|\varphi\right|^{p_{1}}dx+ \int_{\Omega}\chi_{2}^{\epsilon}(x)\left|\varphi\right|^{p_{2}}dx\right.\notag\\
	&\quad\left.+ \int_{\Omega}\chi_{1}^{\epsilon}(x)\left|\Psi\right|^{p_{1}}dx + \int_{\Omega}\chi_{2}^{\epsilon}(x)\left|\Psi\right|^{p_{2}}dx\right)^{\frac{2(p_{2}-p_{1})}{2p_{2}-p_{1}}}\notag\\ 
	&+ \int_{\Omega}\chi_{1}^{\epsilon}(x)\left|\varphi-\Psi\right|^{p_{1}}dx+\left(\int_{\Omega}\chi_{2}^{\epsilon}(x)\left|\varphi-\Psi\right|^{p_{2}}dx\right)^{\frac{1}{p_{2}-1}}\left(\left|\Omega\right|+\int_{\Omega}\chi_{1}^{\epsilon}(x)\left|\varphi\right|^{p_{1}}dx\right.\notag\\
	&\quad\left.\left.+ 
\int_{\Omega}\chi_{2}^{\epsilon}(x)\left|\varphi\right|^{p_{2}}dx + \int_{\Omega}\chi_{1}^{\epsilon}(x)\left|\Psi\right|^{p_{1}}dx + \int_{\Omega}\chi_{2}^{\epsilon}(x)\left|\Psi\right|^{p_{2}}dx\right)^{\frac{p_{2}-2}{p_{2}-1}}\right]\notag
\end{align}	
\end{itemize}
\end{lemma}
	
\begin{lemma}
\label{proofaprioribound}
Let $u_{\epsilon}$ be the solution to (\ref{Dirichlet}).  Then the a priori bound (\ref{aprioribound}) holds.
\end{lemma}

\begin{lemma}
\label{uniform boundedness of p at M}
	If the microstructure is dispersed or layered, we have that $$\sup_{\epsilon>0}\left\{\int_{\Omega}\chi_{i}^{\epsilon}(x)\left|P_{\epsilon}(x,M_{\epsilon}\nabla u(x))\right|^{p_{i}}dx\right\}\leq C<\infty \text{, for $i=1,2$.}$$
\end{lemma}

We use Lemma~\ref{lemma1} to prove structure conditions of $b$ (\ref{b}) in the following lemma.
\begin{lemma}
\label{monconb}
The function $b$, defined in (\ref{b}), satisfies the following structure properties: for every $\xi_{1},\xi_{2}\in\mathbb{R}^{n}$
\begin{enumerate}
\item Monotonicity: 
\begin{equation}
	\label{Monb}
	\displaystyle 
	\left(b(\xi_{2}) - b(\xi_{1}),\xi_{2}-\xi_{1}\right)\geq 0 
\end{equation} 
\item Continuity: There exists a positive constant $C$ such that
\begin{itemize}
	\item For $1<p_{1}\leq p_{2}\leq2$:
\begin{align}
\label{Conb}
	\displaystyle 
	&\left|b(\xi_{1}) - b(\xi_{2})\right|\\
	&\quad\leq C\left[\left|\xi_{2}-\xi_{1}\right|^{\frac{p_{1}-1}{3-p_{1}}}\left(1+\theta_{1}\left|\xi_{1}\right|^{p_{1}}+\theta_{2}\left|\xi_{1}\right|^{p_{2}}+\theta_{1}\left|\xi_{2}\right|^{p_{1}}+\theta_{2}\left|\xi_{2}\right|^{p_{2}}\right)^{\frac{(2-p_{1})(p_{1}-1)}{p_{1}(3-p_{1})}}\right.\notag\\
	&\qquad\left.+ \left|\xi_{2}-\xi_{1}\right|^{\frac{p_{2}-1}{3-p_{2}}}\left(1+\theta_{1}\left|\xi_{1}\right|^{p_{1}}+\theta_{2}\left|\xi_{1}\right|^{p_{2}}+\theta_{1}\left|\xi_{2}\right|^{p_{1}}+\theta_{2}\left|\xi_{2}\right|^{p_{2}}\right)^{\frac{(2-p_{2})(p_{2}-1)}{p_{2}(3-p_{2})}}\right]\notag
\end{align}	
	\item  For $1<p_{1}\leq2\leq p_{2}$:
\begin{align}
\label{Conb2}
	\displaystyle 
	&\left|b(\xi_{1}) - b(\xi_{2})\right|\\
	&\quad\leq C\left[\left|\xi_{2}-\xi_{1}\right|^{\frac{p_{1}-1}{3-p_{1}}}\left(1+\theta_{1}\left|\xi_{1}\right|^{p_{1}}+\theta_{2}\left|\xi_{1}\right|^{p_{2}}+\theta_{1}\left|\xi_{2}\right|^{p_{1}}+\theta_{2}\left|\xi_{2}\right|^{p_{2}}\right)^{\frac{(2-p_{1})(p_{1}-1)}{p_{1}(3-p_{1})}}\right.\notag\\
	&\qquad\left.+ \left|\xi_{2}-\xi_{1}\right|^{\frac{1}{p_{2}-1}}\left(1+\theta_{1}\left|\xi_{1}\right|^{p_{1}}+\theta_{2}\left|\xi_{1}\right|^{p_{2}}+\theta_{1}\left|\xi_{2}\right|^{p_{1}}+\theta_{2}\left|\xi_{2}\right|^{p_{2}}\right)^{\frac{p_{2}-2}{p_{2}-1}}\right]\notag
\end{align}		
\end{itemize}
\end{enumerate}
These structure conditions for $b$ are different to the ones obtained in \cite{Jimenez2010} and their proofs require different techniques, for example, the use of (\ref{conb1}) to obtain (\ref{Conb}).  These structure conditions (\ref{Conb}) and (\ref{Conb2}) will be used in the proof of Theorem~\ref{corrector}.    
\end{lemma}

\begin{lemma}
\label{unifboundDP}
For all $j=0,...,m$, we have that $\displaystyle \int_{\Omega_{j}}\left|\left(A_{\epsilon}\left(x,P_{\epsilon}\left(x,\eta_{j}\right)\right),\nabla u_{\epsilon}(x)\right)\right|dx$ and $\displaystyle \int_{\Omega_{j}}\left|\left(A_{\epsilon}\left(x,\nabla u_{\epsilon}(x)\right),P_{\epsilon}\left(x,\eta_{j}\right)\right)\right|dx$ are uniformly bounded with respect to $\epsilon$.
\end{lemma}

\begin{lemma}
\label{dunfordpettis}
	As $\epsilon\rightarrow0$, up to a subsequence, $\left(A_{\epsilon}\left(\cdot,P_{\epsilon}\left(\cdot,\eta_{j}\right)\right),\nabla u_{\epsilon}(\cdot)\right)$ converges weakly to a function $g_{j}\in L^{1}(\Omega_{j};\mathbb{R})$, for all $j=0,...,m$.  
In a similar way, up to a subsequence, $\left(A_{\epsilon}\left(\cdot,\nabla u_{\epsilon}(\cdot)\right),P_{\epsilon}\left(\cdot,\eta_{j}\right)\right)$ converges weakly to a function $h_{j}\in L^{1}(\Omega_{j};\mathbb{R})$, for all $j=0,...,m$.
\end{lemma}
\section{Proof of Main Results}
\label{main}
\subsection{Proof of the Corrector Theorem}	
\label{proof corrector}
We are now in the position to give the proof of Theorem~\ref{corrector}.  We present the proof for the case when $1<p_{1}\leq p_{2}\leq2$, for $1<p_{1}\leq 2\leq p_{2}$ the proof is very similar and the correspondig formulas can be found in the Appendix, in Section~\ref{appendixcorrector}
\begin{proof}
Let $u_{\epsilon}\in W_{0}^{1,p_{1}}(\Omega)$ be the solutions of (\ref{Dirichlet}).  By (\ref{MonA}), Lemma~\ref{proofaprioribound}, and Lemma~{\ref{uniform boundedness of p at M}} we have that
\begin{align*}
	\displaystyle	
&\int_{\Omega}\left[\chi_{1}^{\epsilon}(x)\left|P_{\epsilon}\left(x,M_{\epsilon}\nabla u(x)\right)-\nabla u_{\epsilon}(x)\right|^{p_{1}} + \chi_{2}^{\epsilon}(x)\left|P_{\epsilon}\left(x,M_{\epsilon}\nabla u(x)\right)-\nabla u_{\epsilon}(x)\right|^{p_{2}}\right]dx\\
	&\leq C\left[\left(\left|\Omega\right|+\int_{\Omega}\chi_{1}^{\epsilon}(x)\left|P_{\epsilon}(x,M_{\epsilon}\nabla u(x))\right|^{p_{1}}dx+\int_{\Omega}\chi_{1}^{\epsilon}(x)\left|\nabla u_{\epsilon}(x)\right|^{p_{1}}dx\right)^{\frac{2-p_{1}}{2}}\right.\\
&\times\left(\int_{\Omega}\chi_{1}^{\epsilon}(x)\left(A_{\epsilon}\left(x,P_{\epsilon}\left(x,M_{\epsilon} \nabla u(x)\right)\right)-A_{\epsilon}\left(x,\nabla u_{\epsilon}(x)\right),P_{\epsilon}\left(x,M_{\epsilon}\nabla u(x)\right)-\nabla u_{\epsilon}(x)\right)dx\right)^{\frac{p_{1}}{2}}\\
&+\left(\int_{\Omega}\chi_{2}^{\epsilon}(x)\left(A_{\epsilon}\left(x,P_{\epsilon}\left(x,M_{\epsilon} \nabla u(x)\right)\right)-A_{\epsilon}\left(x,\nabla u_{\epsilon}(x)\right),P_{\epsilon}\left(x,M_{\epsilon}\nabla u(x)\right)-\nabla u_{\epsilon}(x)\right)dx\right)^{\frac{p_{2}}{2}}\\
&\left.\times\left(\int_{\Omega}\chi_{2}^{\epsilon}(x)\left(1+\left|P_{\epsilon}(x,M_{\epsilon}\nabla u(x))\right|^{p_2}+\left|\nabla u_{\epsilon}(x)\right|\right)^{p_{2}}dx\right)^{\frac{2-p_{2}}{2}} \right]\\
	&\leq C\sum_{i=1}^{2}\left(\int_{\Omega}\left(A_{\epsilon}\left(x,P_{\epsilon}\left(x,M_{\epsilon} \nabla u(x)\right)\right)-A_{\epsilon}\left(x,\nabla u_{\epsilon}(x)\right),P_{\epsilon}\left(x,M_{\epsilon}\nabla u(x)\right)-\nabla u_{\epsilon}(x)\right)dx\right)^{\frac{p_{i}}{2}}
\end{align*}

To prove Theorem~\ref{corrector}, we show that 
\begin{align*}
	\displaystyle
&\int_{\Omega}\left(A_{\epsilon}\left(x,P_{\epsilon}\left(x,M_{\epsilon}\nabla u(x)\right)\right)-A_{\epsilon}\left(x,\nabla u_{\epsilon}(x)\right),P_{\epsilon}\left(x,M_{\epsilon}\nabla u(x)\right)-\nabla u_{\epsilon}(x)\right)dx\\
&=\int_{\Omega}\left(A_{\epsilon}\left(x,P_{\epsilon}\left(x,M_{\epsilon}\nabla u\right)\right),P_{\epsilon}\left(x,M_{\epsilon}\nabla u\right)\right)dx-\int_{\Omega}\left(A_{\epsilon}\left(x,P_{\epsilon}\left(x,M_{\epsilon}\nabla u\right)\right),\nabla u_{\epsilon}\right)dx\\
	&\quad-\int_{\Omega}\left(A_{\epsilon}\left(x,\nabla u_{\epsilon}\right),P_{\epsilon}\left(x,M_{\epsilon}\nabla u\right)\right)dx+\int_{\Omega}\left(A_{\epsilon}\left(x,\nabla u_{\epsilon}\right),\nabla u_{\epsilon}\right)dx
\end{align*}
goes to 0, as $\epsilon\rightarrow0$.  This is done in four steps.
	
In what follows, we use the following notation $$\displaystyle \xi_{\epsilon}^{k}=\frac{1}{\left|Y_{\epsilon}^{k}\right|}\int_{Y_{\epsilon}^{k}}\nabla u dx.$$

\textbf{STEP~1}
\vspace{1mm}

Let us prove that
\begin{equation}
	\label{First}
	\displaystyle
\int_{\Omega}\left(A_{\epsilon}\left(x,P_{\epsilon}\left(x,M_{\epsilon}\nabla u\right)\right),P_{\epsilon}\left(x,M_{\epsilon}\nabla u\right)\right)dx \rightarrow \int_{\Omega}\left(b(\nabla u),\nabla u\right)dx
\end{equation}
as $\epsilon\rightarrow0$.

\begin{proof}
From (\ref{inner product a with p}) and (\ref{approximation}), we obtain
\begin{align*}
	\displaystyle
	& \int_{\Omega}\left(A_{\epsilon}\left(x,P_{\epsilon}\left(x,M_{\epsilon}\nabla u(x)\right)\right),P_{\epsilon}\left(x,M_{\epsilon}\nabla u(x)\right)\right)dx\\
	&\quad= \sum_{k\in I_{\epsilon}}\int_{Y_{\epsilon}^{k}}\left(A\left(\frac{x}{\epsilon},P\left(\frac{x}{\epsilon},\xi_{\epsilon}^{k}\right)\right),P\left(\frac{x}{\epsilon},\xi_{\epsilon}^{k}\right)\right)dx\\
	&\quad=\epsilon^{n}\sum_{k\in I_{\epsilon}}\int_{Y}\left(A\left(y,P\left(y,\xi_{\epsilon}^{k}\right)\right),P\left(y,\xi_{\epsilon}^{k}\right)\right)dy\\
	&\quad= \sum_{k\in I_{\epsilon}}\int_{\Omega}\chi_{Y_{\epsilon}^{k}}(x)\left(b(\xi_{\epsilon}^{k}),\xi_{\epsilon}^{k}\right)dx=\int_{\Omega}\left(b(M_{\epsilon}\nabla u(x)),M_{\epsilon}\nabla u(x)\right)dx.		
\end{align*}

By (\ref{Conb}) in Lemma~\ref{monconb}, H\"older's inequality, Theorem~\ref{regularity}, and Jensen's inequality, we have 
\begin{align*}
	\displaystyle
	&\int_{\Omega}\left|b(M_{\epsilon}\nabla u(x))-b(\nabla u(x))\right|^{q_{1}}dx\\
	&\quad\leq C\left[\left(\int_{\Omega}\left|M_{\epsilon}\nabla u-\nabla u\right|^{p_{2}}dx\right)^{\frac{p_{1}-1}{(p_{2}-1)(3-p_{1})}}+ \left(\int_{\Omega}\left|M_{\epsilon}\nabla u-\nabla u\right|^{p_{2}}dx\right)^{\frac{1}{3-p_{2}}}\right]			
\end{align*}

From Property~1 of $M_{\epsilon}$, we obtain that
\begin{equation}
	\label{step1-1}
	\displaystyle b(M_{\epsilon}\nabla u)\rightarrow b(\nabla u)  \text{  in $L^{q_{1}}(\Omega;\mathbb{R}^{n})$}\text{, as $\epsilon\rightarrow0$}.
\end{equation}	

Now, (\ref{First}) follows from (\ref{step1-1}) since $M_{\epsilon}\nabla u\rightarrow \nabla u$ in $L^{p_2}(\Omega;\mathbb{R}^n)$, so
\begin{align*}
	\displaystyle
	&\int_{\Omega}\left(A_{\epsilon}\left(x,P_{\epsilon}\left(x,M_{\epsilon}\nabla u(x)\right)\right),P_{\epsilon}\left(x,M_{\epsilon}\nabla u(x)\right)\right)dx\\
	&\quad = \int_{\Omega}\left(b(M_{\epsilon}\nabla u(x),M_{\epsilon}\nabla u(x)\right)dx\\
	& \qquad\rightarrow \int_{\Omega}\left(b(\nabla u(x)),\nabla u(x)\right)dx\text{, as $\epsilon\rightarrow0$.}
\end{align*}	
\end{proof}

\textbf{STEP~2}
\vspace{1mm}

We now show that
\begin{equation}
	\label{Second}
	\displaystyle \int_{\Omega}\left(A_{\epsilon}\left(x,P_{\epsilon}\left(x,M_{\epsilon}\nabla u(x)\right)\right),\nabla u_{\epsilon}(x)\right)dx
	\rightarrow \int_{\Omega}\left(b(\nabla u(x)),\nabla u(x)\right)dx 
\end{equation}
as $\epsilon\rightarrow0$.

\begin{proof}
Let $\delta>0$.  From Theorem~\ref{regularity} we have $\nabla u\in L^{p_{2}}(\Omega;\mathbb{R}^{n})$ and there exists 
a simple function $\Psi$ satisfying the assumptions of Lemma~\ref{lemma3} such that 
\begin{equation}
	\label{approximation with simple function of Du}
	\displaystyle
	\left\|\nabla u-\Psi\right\|_{L^{p_{2}}(\Omega;\mathbb{R}^{n})}\leq\delta.
\end{equation}

Let us write
\begin{align*}
	\displaystyle &\int_{\Omega}\left(A_{\epsilon}\left(x,P_{\epsilon}\left(x,M_{\epsilon}\nabla u(x)\right)\right),\nabla u_{\epsilon}(x)\right)dx \\
	&\quad= \int_{\Omega}\left(A_{\epsilon}\left(x,P_{\epsilon}\left(x,\Psi\right)\right),\nabla u_{\epsilon}\right)dx \\
	&\qquad+ \int_{\Omega}\left(A_{\epsilon}\left(x,P_{\epsilon}\left(x,M_{\epsilon}\nabla u\right)\right)-A_{\epsilon}\left(x,P_{\epsilon}\left(x,\Psi \right)\right),\nabla u_{\epsilon}\right)dx.
\end{align*}
We first show that $$\int_{\Omega}\left(A_{\epsilon}\left(x,P_{\epsilon}\left(x,\Psi(x)\right)\right),\nabla u_{\epsilon}(x)\right)dx\rightarrow\int_{\Omega}\left(b(\Psi(x)),\nabla u(x)\right)dx \text{ as $\epsilon\rightarrow0$}.$$  
We have $$\int_{\Omega}\left(A_{\epsilon}\left(x,P_{\epsilon}\left(x,\Psi(x)\right)\right),\nabla u_{\epsilon}(x)\right)dx=\sum_{j=0}^{m}\int_{\Omega_{j}}\left(A_{\epsilon}\left(x,P_{\epsilon}\left(x,\eta_{j}\right)\right),\nabla u_{\epsilon}(x)\right)dx.$$
	
Now from (\ref{p5}), we have that $\displaystyle A_{\epsilon}\left(\cdot,P_{\epsilon}\left(\cdot,\eta_{j}\right)\right)\rightharpoonup b(\eta_{j})\in L^{q_{2}}(\Omega_{j};\mathbb{R}^{n}),$ and by (\ref{div with p}), $\displaystyle \int_{\Omega_{j}}\left(A_{\epsilon}\left(x,P_{\epsilon}\left(x,\eta_{j}\right)\right),\nabla\varphi(x)\right)dx=0,$ for $\varphi\in W_{0}^{1,p_{1}}(\Omega_{j})$.

Take $\varphi=\delta u_{\epsilon}$, with $\delta\in C_{0}^{\infty}(\Omega_{j})$ to get $$0=\int_{\Omega_{j}}\left(A_{\epsilon}\left(x,P_{\epsilon}\left(x,\eta_{j}\right)\right),(\nabla\delta)u_{\epsilon}\right)dx + \int_{\Omega_{j}}\left(A_{\epsilon}\left(x,P_{\epsilon}\left(x,\eta_{j}\right)\right),(\nabla u_{\epsilon})\delta\right)dx.$$
	
Taking the limit as $\epsilon\rightarrow0$, and using the fact that $u^{\epsilon}\rightharpoonup u$ in $W_{0}^{1,p_{1}}(\Omega)$ 
and (\ref{p5}), we have by Lemma~\ref{dunfordpettis} that $$\int_{\Omega_{j}}g_{j}(x)\delta(x)dx =\lim_{\epsilon\rightarrow0}\int_{\Omega_{j}}\left(A_{\epsilon}\left(x,P_{\epsilon}\left(x,\eta_{j}\right)\right),(\nabla u_{\epsilon})\delta\right)dx=\int_{\Omega_{j}}\left(b(\eta_{j}),(\nabla u)\delta\right)dx$$

Therefore, we may conclude that $g_{j}=\left(b(\eta_{j}),\nabla u\right)$, so $$\sum_{j=0}^{n}\int_{\Omega_{j}}\left(A_{\epsilon}\left(x,P_{\epsilon}\left(x,\eta_{j}\right)\right),\nabla u_{\epsilon}(x)\right)dx\rightarrow\sum_{j=0}^{n}\int_{\Omega_{j}}\left(b(\eta_{j}),\nabla u(x)\right)dx\text{, as $\epsilon\rightarrow0$.}$$ 

Thus, we get $$\int_{\Omega}\left(A_{\epsilon}\left(x,P_{\epsilon}\left(x,\Psi(x)\right)\right),\nabla u_{\epsilon}(x)\right)dx\rightarrow\int_{\Omega}\left(b(\Psi(x)),\nabla u(x)\right)dx\text{, as $\epsilon\rightarrow0$.}$$ 

On the other hand, let us estimate $$\int_{\Omega}\left(A_{\epsilon}\left(x,P_{\epsilon}\left(x,M_{\epsilon}\nabla u(x)\right)\right)-A_{\epsilon}\left(x,P_{\epsilon}\left(x,\Psi(x)\right)\right),\nabla u_{\epsilon}(x)\right)dx.$$  
	
By (\ref{ConA}), H\"older's inequality, and (\ref{aprioribound}) we obtain
\begin{align}
	\label{step2-1}
	\displaystyle 
&\left|\int_{\Omega}\left(A_{\epsilon}\left(x,P_{\epsilon}\left(x,M_{\epsilon}\nabla u(x)\right)\right)-A_{\epsilon}\left(x,P_{\epsilon}\left(x,\Psi(x) \right)\right),\nabla u_{\epsilon}(x)\right)dx\right|\\
	&\quad\leq C\left[\left(\int_{\Omega}\chi_{1}^{\epsilon}(x)\left|P_{\epsilon}\left(x,M_{\epsilon}\nabla u\right)-P_{\epsilon}\left(x,\Psi \right)\right|^{p_{1}}dx\right)^{\frac{p_{1}-1}{p_{1}}}\right.\notag\\
&\qquad\left.+\left(\int_{\Omega}\chi_{2}^{\epsilon}(x)\left|P_{\epsilon}\left(x,M_{\epsilon}\nabla u\right)-P_{\epsilon}\left(x,\Psi \right)\right|^{p_{2}}dx\right)^{\frac{p_{2}-1}{p_{2}}}\right]\notag
\end{align}

Applying Lemma~\ref{lemma3} and (\ref{approximation with simple function of Du}) to (\ref{step2-1}), we discover that	
\begin{align}
	\label{step2-3}
	\displaystyle & \limsup_{\epsilon\rightarrow0}\left|\int_{\Omega}\left(A_{\epsilon}\left(x,P_{\epsilon}\left(x,M_{\epsilon}\nabla u(x)\right)\right)-A_{\epsilon}\left(x,P_{\epsilon}\left(x,\Psi(x)\right)\right),\nabla u_{\epsilon}(x)\right)dx\right|\\
	&\quad \leq C\sum_{i=1}^{2}\left[\delta^{\frac{p_{1}}{3-p_{1}}}+\delta^{\frac{p_{1}p_{2}}{2p_{2}-p_{2}p_{1}+p_{1}}}+\delta^{\frac{p_{1}p_{2}}{2p_{1}-p_{2}p_{1}+p_{2}}}+\delta^{\frac{p_{2}}{3-p_{2}}}\right]^{\frac{p_{i}-1}{p_{i}}},\notag
\end{align}
where $C$ is independent of $\delta$.  Since $\delta$ is arbitrary we conclude that the limit on the left hand side of (\ref{step2-3}) is equal to $0$.

Finally, using the continuity of $b$ (\ref{Conb}) in Lemma~\ref{monconb}, Theorem~\ref{regularity}, and H\"older's inequality, we obtain $$\left|\int_{\Omega}\left(b(\nabla u(x))-b(\Psi(x)),\nabla u(x)\right)dx\right|\leq C\left[\delta^{\frac{p_{2}(p_{1}-1)}{(p_{2}-1)(3-p_{1})}}+\delta^{\frac{p_{2}}{3-p_{2}}}\right]^{\frac{1}{q_{1}}},$$
where $C$ does not depend on $\delta$. 

Step~2 is proved noticing that $\delta$ can be taken arbitrarily small.
\end{proof}

\textbf{STEP~3}
\vspace{1mm}
	
We will show that
\begin{equation}
\label{Third}
	\displaystyle\int_{\Omega}\left(A_{\epsilon}\left(x,\nabla u_{\epsilon}(x)\right),P_{\epsilon}\left(x,M_{\epsilon}\nabla u(x)\right)\right)dx
	 \rightarrow\int_{\Omega}\left(b(\nabla u(x)),\nabla u(x)\right)dx
\end{equation}
as $\epsilon\rightarrow0$.	

\begin{proof}
Let $\delta>0$.  As in the proof of Step~2, assume $\Psi$ is a simple function satisfying assumptions of 
Lemma~\ref{lemma3} and such that $\displaystyle \left\|\nabla u-\Psi\right\|_{L^{p_{2}}(\Omega;\mathbb{R}^{n})}<\delta$.
		
Let us write 
\begin{align*}
	\displaystyle
	&\int_{\Omega}\left(A_{\epsilon}\left(x,\nabla u_{\epsilon}(x)\right),P_{\epsilon}\left(x,M_{\epsilon}\nabla u(x)\right)\right)dx\\
	&\quad=\int_{\Omega}\left(A_{\epsilon}\left(x,\nabla u_{\epsilon}(x)\right),P_{\epsilon}\left(x,\Psi(x)\right)\right)dx\\
	&\qquad+\int_{\Omega}\left(A_{\epsilon}\left(x,\nabla u_{\epsilon}(x)\right),P_{\epsilon}\left(x,M_{\epsilon}\nabla u(x)\right)-P_{\epsilon}\left(x,\Psi(x)\right)\right)dx.	
\end{align*}
	
We first show that $$\int_{\Omega}\left(A_{\epsilon}\left(x,\nabla u_{\epsilon}(x)\right),P_{\epsilon}\left(x,\Psi(x)\right)\right)dx\rightarrow\int_{\Omega}\left(b\left(\nabla u(x)\right),\Psi(x)\right)dx.$$

We start by writing $$\int_{\Omega}\left(A_{\epsilon}\left(x,\nabla u_{\epsilon}(x)\right),P_{\epsilon}\left(x,\Psi(x)\right)\right)dx= \sum_{j=0}^{m}\int_{\Omega_{j}}\left(A_{\epsilon}\left(x,\nabla u_{\epsilon}(x)\right),P_{\epsilon}\left(x,\eta_{j}\right)\right)dx.$$

From Lemma~\ref{dunfordpettis}, up to a subsequence, $\left(A_{\epsilon}\left(\cdot,\nabla u_{\epsilon}\right),P_{\epsilon}\left(\cdot,\eta_{j}\right)\right)$ converges weakly to a function $h_{j}\in L^{1}(\Omega_{j};\mathbb{R})$, as $\epsilon\rightarrow0$.

By Theorem~\ref{homogenization}, we have $\displaystyle A_{\epsilon}\left(\cdot,\nabla u_{\epsilon}\right)\rightharpoonup b(\nabla u)\in L^{q_{2}}(\Omega;\mathbb{R}^{n})$ and $$\displaystyle -\mbox{div}\left(A_{\epsilon}\left(x,\nabla u_{\epsilon}\right)\right)=f=-\mbox{div}\left(b(\nabla u)\right).$$ 
	
From (\ref{p3}), $p_{\epsilon}$ satisfies $\displaystyle P_{\epsilon}(\cdot,\eta_{j})\rightharpoonup\eta_{j}$ in $L^{p_{1}}(\Omega_{j},\mathbb{R}^{n})$. 
	
Arguing as in Step~2, we find that $\displaystyle \left(A_{\epsilon}\left(x,\nabla u_{\epsilon}(x)\right),P_{\epsilon}\left(x,\eta_{j}\right)\right) \rightharpoonup \left(b(\nabla u(x)),\eta_{j}\right)$ in $D^{'}(\Omega_{j})$, as $\epsilon\rightarrow0$.

Therefore, we may conclude that $h_{j}=\left(b(\nabla u),\eta_{j}\right)$, and hence, $$\sum_{j=0}^{n}\int_{\Omega_{j}}\left(A_{\epsilon}\left(x,\nabla u_{\epsilon}(x)\right),P_{\epsilon}\left(x,\eta_{j}\right)\right)dx\rightarrow\sum_{j=0}^{n}\int_{\Omega_{j}}\left(b(\nabla u(x)),\eta_{j}\right)dx\text{, as $\epsilon\rightarrow0$.}$$ 

Thus, we get $$\displaystyle \int_{\Omega}\left(A_{\epsilon}\left(x,\nabla u_{\epsilon}(x)\right),P_{\epsilon}\left(x,\Psi(x)\right)\right)dx\rightarrow\int_{\Omega}\left(b(\nabla u(x)),\Psi(x)\right)dx\text{, as $\epsilon\rightarrow0.$}$$

Moreover, applying H\"older's inequality, (\ref{ConA}), and (\ref{aprioribound}) we have 
\begin{align*}
	\displaystyle
	& \left|\int_{\Omega}\left(A_{\epsilon}\left(x,\nabla u_{\epsilon}(x)\right),P_{\epsilon}\left(x,M_{\epsilon}\nabla u(x)\right)-P_{\epsilon}\left(x,\Psi(x)\right)\right)dx\right|\\
	&\quad\leq C\sum_{i=1}^{2}\left(\int_{\Omega}\chi_{i}^{\epsilon}(x)\left|P_{\epsilon}\left(x,M_{\epsilon}\nabla u\right)-P_{\epsilon}\left(x,\Psi\right)\right|^{p_{i}}dx\right)^{\frac{1}{p_{i}}}
\end{align*}
	
As in the proof of Step~2 we see that 
\begin{align*}
&\limsup_{\epsilon\rightarrow0}\left|\int_{\Omega}\left(A_{\epsilon}\left(x,\nabla u_{\epsilon}\right),P_{\epsilon}\left(x,M_{\epsilon}\nabla u\right)-P_{\epsilon}\left(x,\Psi\right)\right)dx\right|\\
	&\quad\leq C\sum_{i=1}^{2}\left[\delta^{\frac{p_{1}}{3-p_{1}}}+\delta^{\frac{p_{1}p_{2}}{2p_{2}-p_{2}p_{1}+p_{1}}}+\delta^{\frac{p_{1}p_{2}}{2p_{1}-p_{2}p_{1}+p_{2}}}+\delta^{\frac{p_{2}}{3-p_{2}}}\right]^{\frac{1}{p_{i}}},
\end{align*} 	
where $C$ does not depend on $\delta$.

Hence, proceeding as in Step~2, we find that
\begin{align*}
	\displaystyle				
	& \limsup_{\epsilon\rightarrow0}\left|\int_{\Omega}\left(A_{\epsilon}\left(x,\nabla u_{\epsilon}\right),P_{\epsilon}\left(x,M_{\epsilon} \nabla u\right)\right)dx-\int_{\Omega}\left(b(\nabla u),\nabla u\right)dx\right|\\
	&\leq C\left[\sum_{i=1}^{2}\left(\delta^{\frac{p_{1}}{3-p_{1}}}+\delta^{\frac{p_{1}p_{2}}{2p_{2}-p_{2}p_{1}+p_{1}}}+\delta^{\frac{p_{1}p_{2}}{2p_{1}-p_{2}p_{1}+p_{2}}}+\delta^{\frac{p_{2}}{3-p_{2}}}\right)^{\frac{1}{p_{i}}}+\delta\left\|b(\nabla u)\right\|_{L^{q_{2}}(\Omega,\mathbb{R}^{n})}\right],
\end{align*}
where $C$ is independent of $\delta$.  Now since $\delta$ is arbitrarily small, the proof of Step~3 is complete.
\end{proof}

\textbf{STEP~4}
\vspace{1mm}
	
Finally, let us prove that
\begin{equation}
	\label{Fourth}
	\displaystyle \int_{\Omega}\left(A_{\epsilon}\left(x,\nabla u_{\epsilon}(x)\right),\nabla u_{\epsilon}(x)\right)dx
	 \rightarrow\int_{\Omega}\left(b(\nabla u(x)),\nabla u(x)\right)dx \text{, as $\epsilon\rightarrow0$}.
\end{equation}
	
\begin{proof}
	Since
	\begin{equation}
		\displaystyle \int_{\Omega}\left(A_{\epsilon}\left(x,\nabla u_{\epsilon}\right),\nabla u_{\epsilon}\right)dx=\left\langle -\mbox{div}\left(A_{\epsilon}\left(x,\nabla u_{\epsilon}\right)\right),u_{\epsilon}\right\rangle=\left\langle f,u_{\epsilon}\right\rangle,
	\end{equation}	
	\begin{equation}
		\displaystyle \int_{\Omega}\left(b(\nabla u),\nabla u\right)dx=\left\langle -\mbox{div}\left(b\left(\nabla u\right)\right),u\right\rangle=\left\langle f,u\right\rangle,
	\end{equation}
and $u_{\epsilon}\rightharpoonup u$ in $W^{1,p_{1}}(\Omega)$, the result follows immediately.
\end{proof}	

Finally, Theorem~\ref{corrector} follows from (\ref{First}), (\ref{Second}), (\ref{Third}) and (\ref{Fourth}).
\end{proof}

\subsection{Proof of the Lower Bound on the Amplification of the Macroscopic Field by the Microstructure}
\label{proof fluctuations}

The sequence $\left\{\chi_{i}^{\epsilon}(x)\nabla u_{\epsilon}(x)\right\}_{\epsilon>0}$ has a Young measure 
$\nu^{i}=\left\{\nu_{x}^{i}\right\}_{x\in\Omega}$ associated to it (see Theorem~6.2 and the discussion following 
in \cite{Pedregal1997}), for $i=1,2$.

As a consequence of Theorem~\ref{corrector} proved in the previous section, we have that  $$\left\|\chi_{i}^{\epsilon}(x)P\left(\frac{x}{\epsilon},M_{\epsilon}(\nabla u)(x)\right)- \chi_{i}^{\epsilon}(x)\nabla u_{\epsilon}(x)\right\|_{\textbf{L}^{p_{i}}(\Omega;\mathbb{R}^{n})}\rightarrow0,$$
as $\epsilon\rightarrow0$, which implies that the sequences$$\left\{\chi_{i}^{\epsilon}(x)P\left(\frac{x}{\epsilon},M_{\epsilon}(\nabla u)(x)\right)\right\}_{\epsilon>0} \text{  and  } \left\{\chi_{i}^{\epsilon}(x)\nabla u_{\epsilon}(x)\right\}_{\epsilon>0}$$share the same Young measure (see Lemma~6.3 of \cite{Pedregal1997}), for $i=1,2$.
  
The next lemma identifies the Young measure $\nu^i$.  The lemma is proven for $1<p_1\leq p_2\leq2$; the proof for the case when $1<p_1\leq2\leq p_2$ follows in a similar way.   
\begin{lemma}
\label{lemmafluctuations}
For all $\phi\in C_{0}(\mathbb{R}^{n})$ and for all $\zeta\in C^{\infty}_{0}(\mathbb{R}^{n})$, we have
\begin{equation}
	\label{field-1}
 	\displaystyle
\int_{\Omega}\zeta(x)\int_{\textbf{R}^{n}}\phi(\lambda)d\nu_{x}^{i}(\lambda)dx=\int_{\Omega}\zeta(x)\int_{Y}\phi(\chi_{i}(y)P(y,\nabla u(x)))dydx
\end{equation}
\end{lemma}
\begin{proof}
To prove (\ref{field-1}), we will show that given $\phi\in C_0(\mathbb{R}^n)$  and $\zeta\in C_0^\infty(\mathbb{R}^n)$,
\begin{align}
	\label{field-2}
&\lim_{\epsilon\rightarrow0}\int_{\Omega}\zeta(x)\phi\left(\chi_{i}^{\epsilon}(x)P\left(\frac{x}{\epsilon},M_{\epsilon}\left(\nabla u\right)(x)\right)\right)dx\notag\\
	 &\quad=\int_{\Omega}\zeta(x)\int_{Y}\phi(\chi_{i}(y)P(y,\nabla u(x)))dydx.
\end{align}
We consider the difference
\begin{align}
	\label{difference}
	\displaystyle 
	& \left|\int_{\Omega}\zeta(x)\phi\left(\chi_{i}\left(\frac{x}{\epsilon}\right)P\left(\frac{x}{\epsilon},M_{\epsilon}(\nabla u)(x)\right)\right)dx - \int_{\Omega}\zeta(x)\int_{Y}\phi\left(\chi_{i}\left(y\right)P\left(y,\nabla u(x)\right)\right)dydx\right|\notag\\
	&\leq\left|\sum_{k\in I_{\epsilon}}\int_{Y_{\epsilon}^{k}}\zeta(x)\phi\left(\chi_{i}\left(\frac{x}{\epsilon}\right)P\left(\frac{x}{\epsilon},\xi_{\epsilon}^{k}\right)\right)dx - \int_{\Omega_{\epsilon}}\zeta(x)\int_{Y}\phi\left(\chi_{i}\left(y\right)P\left(y,\nabla u(x)\right)\right)dydx\right|\notag\\
	&\quad+C\left|\Omega\setminus\Omega_{\epsilon}\right|.
\end{align}
Note that the term $C\left|\Omega\setminus\Omega_{\epsilon}\right|$ goes to $0$, as $\epsilon\rightarrow0$.  Now set $x_{\epsilon}^{k}$ to be the center of $Y_{\epsilon}^{k}$. On the first integral use the change of variables $x=x_{\epsilon}^{k}+\epsilon y$, where $y$ belongs to $Y$, and since $dx=\epsilon^{n}dy$, we get
\begin{align*}
	\displaystyle 
	& \left|\sum_{k\in I_{\epsilon}}\int_{Y_{\epsilon}^{k}}\zeta(x)\phi\left(\chi_{i}\left(\frac{x}{\epsilon}\right)P\left(\frac{x}{\epsilon},\xi_{\epsilon}^{k}\right)\right)dx - \sum_{k\in I_{\epsilon}}\int_{Y_{\epsilon}^{k}}\zeta(x)\int_{Y}\phi\left(\chi_{i}\left(y\right)P\left(y,\nabla u(x)\right)\right)dydx\right|\\
	&=\left|\sum_{k\in I_{\epsilon}}\epsilon^{n}\int_{Y}\zeta(x_{\epsilon}^{k}+\epsilon y)\phi\left(\chi_{i}\left(y\right)P\left(y,\xi_{\epsilon}^{k}\right)\right)dy\right.\\
	&\quad\left. - \sum_{k\in I_{\epsilon}}\int_{Y_{\epsilon}^{k}}\zeta(x)\int_{Y}\phi\left(\chi_{i}\left(y\right)P\left(y,\nabla u(x)\right)\right)dydx\right|
	&\intertext{Applying Taylor's expansion for $\zeta$, we have}
	&\leq \left|\sum_{k\in I_{\epsilon}}\int_{Y_{\epsilon}^{k}}\int_{Y}\left(\zeta(x)+CO(\epsilon)\right)\left[\phi\left(\chi_{i}\left(y\right)P\left(y,\xi_{\epsilon}^{k}\right)\right)- \phi\left(\chi_{i}\left(y\right)P\left(y,\nabla u(x)\right)\right)\right]dydx\right|\\
	&\quad+CO(\epsilon)\\
&\leq\left|\int_{\Omega_{\epsilon}}\left|\zeta(x)\right|\int_{Y}\left|\phi\left(\chi_{i}\left(y\right)P\left(y,M_{\epsilon}\nabla u(x)\right)\right)- \phi\left(\chi_{i}\left(y\right)P\left(y,\nabla u(x)\right)\right)\right|dydx\right|\\
	&\quad+CO(\epsilon)
	&\intertext{Because of the uniform Lipschitz continuity of $\phi$, we get}
	&\leq C\left|\int_{\Omega_{\epsilon}}\left|\zeta(x)\right|\int_{Y}\left|P\left(y,M_{\epsilon}\nabla u(x)\right)- P\left(y,\nabla u(x)\right)\right|dydx\right| + CO(\epsilon)
	&\intertext{By H\"older's inequality twice and Lemma~\ref{lemma2}, we have}	
	&\leq C\left\{\left[\int_{\Omega_{\epsilon}}\left|M_{\epsilon}\nabla u(x)-\nabla u(x)\right|^{\frac{p_{1}}{3-p_{1}}}\right.\right.\\
	&\quad \times \left(1+\left|M_{\epsilon}\nabla u(x)\right|^{p_{1}}+\left|M_{\epsilon} \nabla u(x)\right|^{p_{2}}+\left|\nabla u(x)\right|^{p_{1}}+\left|\nabla u(x)\right|^{p_{2}}\right)^{\frac{2-p_{1}}{3-p_{1}}}dx\\
	&+ \int_{\Omega_{\epsilon}}\left|M_{\epsilon}\nabla u(x)-\nabla u(x)\right|^{\frac{p_{1}p_{2}}{2p_{2}-p_{1}p_{2}+p_{1}}}\\
	&\quad \times\left(1+\left|M_{\epsilon}\nabla u(x)\right|^{p_{1}}+\left|M_{\epsilon}\nabla u(x)\right|^{p_{2}}+\left|\nabla u(x)\right|^{p_{1}}+\left|\nabla u(x)\right|^{p_{2}}\right)^{\frac{p_{2}(2-p_{1})}{2p_{2}-p_{1}p_{2}+p_{1}}}dx\\
	&+ \int_{\Omega_{\epsilon}}\left|M_{\epsilon}\nabla u(x)-\nabla u(x)\right|^{\frac{p_{1}p_{2}}{2p_{1}-p_{1}p_{2}+p_{2}}}\\
	&\quad \times \left(1+\left|M_{\epsilon}\nabla u(x)\right|^{p_{1}}+\left|M_{\epsilon}\nabla u(x)\right|^{p_{2}}+\left|\nabla u(x)\right|^{p_{1}}+\left|\nabla u(x)\right|^{p_{2}}\right)^{\frac{p_{1}(2-p_{2})}{2p_{1}-p_{1}p_{2}+p_{2}}}dx\\	
	&+ \int_{\Omega_{\epsilon}}\left|M_{\epsilon}\nabla u(x)-\nabla u(x)\right|^{\frac{p_{2}}{3-p_{2}}}\\
	&\quad \times \left.\left(1+\left|M_{\epsilon}\nabla u(x)\right|^{p_{1}}+\left|M_{\epsilon}\nabla u(x)\right|^{p_{2}}+\left|\nabla u(x)\right|^{p_{1}}+\left|\nabla u(x)\right|^{p_{2}}\right)^{\frac{2-p_{2}}{3-p_{2}}}dx\right]^{1/p_{1}}\\
	&+ \left[\int_{\Omega_{\epsilon}}\left|M_{\epsilon}\nabla u(x)-\nabla u(x)\right|^{\frac{p_{1}}{3-p_{1}}}\right.\\
	&\quad \times \left(1+\left|M_{\epsilon}\nabla u(x)\right|^{p_{1}}+\left|M_{\epsilon} \nabla u(x)\right|^{p_{2}}+\left|\nabla u(x)\right|^{p_{1}}+\left|\nabla u(x)\right|^{p_{2}}\right)^{\frac{2-p_{1}}{3-p_{1}}}dx\\
	&+ \int_{\Omega_{\epsilon}}\left|M_{\epsilon}\nabla u(x)-\nabla u(x)\right|^{\frac{p_{1}p_{2}}{2p_{2}-p_{1}p_{2}+p_{1}}}\\
	&\quad \times\left(1+\left|M_{\epsilon}\nabla u(x)\right|^{p_{1}}+\left|M_{\epsilon}\nabla u(x)\right|^{p_{2}}+\left|\nabla u(x)\right|^{p_{1}}+\left|\nabla u(x)\right|^{p_{2}}\right)^{\frac{p_{2}(2-p_{1})}{2p_{2}-p_{1}p_{2}+p_{1}}}dx\\
	&+ \int_{\Omega_{\epsilon}}\left|M_{\epsilon}\nabla u(x)-\nabla u(x)\right|^{\frac{p_{1}p_{2}}{2p_{1}-p_{1}p_{2}+p_{2}}}\\
	&\quad \times \left(1+\left|M_{\epsilon}\nabla u(x)\right|^{p_{1}}+\left|M_{\epsilon}\nabla u(x)\right|^{p_{2}}+\left|\nabla u(x)\right|^{p_{1}}+\left|\nabla u(x)\right|^{p_{2}}\right)^{\frac{p_{1}(2-p_{2})}{2p_{1}-p_{1}p_{2}+p_{2}}}dx\\	
	&+ \int_{\Omega_{\epsilon}}\left|M_{\epsilon}\nabla u(x)-\nabla u(x)\right|^{\frac{p_{2}}{3-p_{2}}}\\
	&\quad \times \left.\left.\left(1+\left|M_{\epsilon}\nabla u(x)\right|^{p_{1}}+\left|M_{\epsilon}\nabla u(x)\right|^{p_{2}}+\left|\nabla u(x)\right|^{p_{1}}+\left|\nabla u(x)\right|^{p_{2}}\right)^{\frac{2-p_{2}}{3-p_{2}}}dx\right]^{1/p_{2}}\right\}\\
	&\qquad + CO(\epsilon).
	&\intertext{Applying H\"older's inequality, Jensen's inequality and Theorem~\ref{regularity}, we get} 
	&\leq C\left\{\left[\left(\int_{\Omega_{\epsilon}}\left|M_{\epsilon}\nabla u(x)-\nabla u(x)\right|^{p_{1}}dx\right)^{\frac{1}{3-p_{1}}}+ \left(\int_{\Omega_{\epsilon}}\left|M_{\epsilon}\nabla u(x)-\nabla u(x)\right|^{p_{2}}dx\right)^{\frac{p_{1}}{2p_{2}-p_{1}p_{2}+p_{1}}}\right.\right.\\
	&\quad \left.+ \left(\int_{\Omega_{\epsilon}}\left|M_{\epsilon}\nabla u(x)-\nabla u(x)\right|^{p_{1}}dx\right)^{\frac{p_{2}}{2p_{1}-p_{1}p_{2}+p_{2}}}\left(\int_{\Omega_{\epsilon}}\left|M_{\epsilon}\nabla u(x)-\nabla u(x)\right|^{p_{2}}dx\right)^{\frac{1}{3-p_{2}}}\right]^{1/p_{1}}\\
	&+ \left[\left(\int_{\Omega_{\epsilon}}\left|M_{\epsilon}\nabla u(x)-\nabla u(x)\right|^{p_{1}}dx\right)^{\frac{1}{3-p_{1}}}+ \left(\int_{\Omega_{\epsilon}}\left|M_{\epsilon}\nabla u(x)-\nabla u(x)\right|^{p_{2}}dx\right)^{\frac{p_{1}}{2p_{2}-p_{1}p_{2}+p_{1}}}\right.
	\intertext{}
	&\quad \left.\left.+ \left(\int_{\Omega_{\epsilon}}\left|M_{\epsilon}\nabla u(x)-\nabla u(x)\right|^{p_{1}}dx\right)^{\frac{p_{2}}{2p_{1}-p_{1}p_{2}+p_{2}}}+ \left(\int_{\Omega_{\epsilon}}\left|M_{\epsilon}\nabla u(x)-\nabla u(x)\right|^{p_{2}}dx\right)^{\frac{1}{3-p_{2}}}\right]^{1/p_{2}}\right\}\\
	&\qquad+ CO(\epsilon).
\end{align*}

Finally, from the approximation property of $M_{\epsilon}$ in Section~\ref{CorrectorSection}, as $\epsilon\rightarrow0$, we obtain (\ref{field-2}).
\end{proof}

\begin{lemma}(See Theorem~6.11 in \cite{Pedregal1997})
\label{pedregal}
The sequence $\left\{\chi_{i}^{\epsilon}(x)\nabla u_{\epsilon}(x)\right\}_{\epsilon>0}$ and the Young measure 
$\nu^{i}=\left\{\nu_{x}^{i}\right\}_{x\in\Omega}$ associated to it satisfy
$$\int_{D}\int_{\mathbb{R}^{n}}\psi\left(x,\lambda\right)d\nu_{x}^{i}(\lambda)dx\leq\liminf_{\epsilon\rightarrow0}\int_{D}\psi\left(x,\chi_{i}^{\epsilon}(x)\nabla u_{\epsilon}(x)\right)dx,$$ for all Carath\'{e}odory functions $\psi\geq0$ and measurable subset $D\subset\Omega$.
\end{lemma}

Using Lemma~\ref{lemmafluctuations} and Lemma~\ref{pedregal}, we obtain
\begin{align*}
 	\displaystyle
\lim_{\epsilon\rightarrow0}\int_{\Omega}\zeta(x)\phi\left(\chi_{i}^{\epsilon}(x)P\left(\frac{x}{\epsilon},M_{\epsilon}(\nabla u)(x)\right)\right)dx&=\int_{\Omega}\zeta(x)\int_{Y}\phi(\chi_{i}(y)P(y,\nabla u(x)))dydx\\  	
& =\int_{\Omega}\zeta(x)\int_{\mathbb{R}^{n}}\phi(\lambda)d\nu_{x}^{i}(\lambda)dx\\
 	& \leq \lim_{\epsilon\rightarrow0}\int_{\Omega}\zeta(x)\phi\left(\chi_{i}^{\epsilon}(x)\nabla u_{\epsilon}(x)\right)dx, 
\end{align*}
for all $\phi\in C_{0}(\mathbb{R}^{n})$ and for all $\zeta\in C^{\infty}_{0}(\mathbb{R}^{n})$, which concludes the proof of Theorem~\ref{fluctuations}.

\section{Appendix}
The proofs presented here are for the case when $1<p_{1}\leq p_{2}\leq2$.  The proofs for the lemmas for the case when $1<p_{1}\leq2\leq p_{2}$ follow in a similar way.  The letter $C$ will represent a generic positive constant independent of $\epsilon$, and it can take different values from one line to the other.
 
\subsection{Proof of Lemma~\ref{lemma1}}
Let $\xi\in\mathbb{R}^{n}$.  By (\ref{MonA}) we have that 
\begin{equation*}
\left(A(y,P(y,\xi)),P(y,\xi)\right)\geq C\left(\chi_{1}(y)\left|P(y,\xi)\right|^{p_{1}} + \chi_{2}(y)\left|P(y,\xi)\right|^{p_{2}}\right).
\end{equation*} 
Integrating both sides over $Y$, using (\ref{ConA}), and Young's inequality, we get
\begin{align*}
\displaystyle
	&\int_{Y}\chi_{1}(y)\left|P(y,\xi)\right|^{p_{1}}dy + \int_{Y}\chi_{2}(y)\left|P(y,\xi)\right|^{p_{2}}dy\\
	&\quad\leq C\left[\frac{\delta^{q_{2}}\displaystyle\int_{Y}\chi_{1}(y)\left|P(y,\xi)\right|^{p_{1}}dy}{q_{2}} + \frac{\delta^{-p_{1}}\displaystyle\int_{Y}\chi_{1}(y)\left|\xi\right|^{p_{1}}dy}{p_{1}}\right.\\
&\qquad\left.+\frac{\delta^{q_{1}}\displaystyle\int_{Y}\chi_{2}(y)\left|P(y,\xi)\right|^{p_{2}}dy}{q_{1}} + \frac{\delta^{-p_{2}}\displaystyle\int_{Y}\chi_{2}(y)\left|\xi\right|^{p_{2}}dy}{p_{2}}\right]
\end{align*}
		
Doing some algebraic manipulations, we obtain 
\begin{align*}
	\displaystyle
\left(1-C\delta^{q_{2}}\right)\int_{Y}\chi_{1}(y)\left|P(y,\xi)\right|^{p_{1}}dy &+ \left(1-C\delta^{q_{1}}\right)\int_{Y}\chi_{2}(y)\left|P(y,\xi)\right|^{p_{2}}dy\\
	&\quad\leq C\left[\left(\delta^{-p_{1}}\left|\xi\right|^{p_{1}}\theta_{1}+\delta^{-p_{2}}\left|\xi\right|^{p_{2}}\theta_{2}\right)\right]
\end{align*}
On choosing an appropiate $\delta$, we finally obtain (\ref{Lemma 1}).

\subsection{Proof of Lemma~\ref{lemma2}}
Using H\"older's inequality, we have
\begin{align*}
	\displaystyle
	& \int_{Y}\chi_{1}(y)\left|P(y,\xi_{1})-P(y,\xi_{2})\right|^{p_{1}}dy+\int_{Y}\chi_{2}(y)\left|P(y,\xi_{1})-P(y,\xi_{2})\right|^{p_{2}}dy\\
	&\leq C\sum_{i=1}^{2}\left[\left(\int_{Y}\frac{\chi_{i}(y)\left|P(y,\xi_{1})-P(y,\xi_{2})\right|^{2}}{\left(1+\left|P(y,\xi_{1})\right|+\left|P(y,\xi_{2})\right|\right)^{2-p_{i}}}dy\right)^{\frac{p_{i}}{2}}\right.\\
&\quad\times\left.\left(\int_{Y}\chi_{i}(y)\left(1+\left|P(y,\xi_{1})\right|+\left|P(y,\xi_{2})\right|\right)^{p_{i}}dy\right)^{\frac{2-p_{i}}{2}}\right]
	&\intertext{By Lemma~\ref{lemma1}, (\ref{MonA}), and (\ref{cell}) we get}
	& \leq C\sum_{i=1}^{2}\left[\left(\int_{Y}\left|A(y,P(y,\xi_{1}))-A(y,P(y,\xi_{2}))\right|\left|\xi_{1}-\xi_{2}\right|dy\right)^{\frac{p_{i}}{2}}\right.\\
	&\qquad\times\left.		\left(1+\left|\xi_{1}\right|^{p_{1}}\theta_{1}+\left|\xi_{2}\right|^{p_{1}}\theta_{1}+\left|\xi_{1}\right|^{p_{2}}\theta_{2}+\left|\xi_{2}\right|^{p_{2}}\theta_{2}\right)^{\frac{2-p_{i}}{2}}\right]
	&\intertext{By (\ref{ConA}) and H\"older's inequality, we have}	
	&\leq C\left[\left(\int_{Y}\chi_{1}(y)\left|P(y,\xi_{1})-P(y,\xi_{2})\right|^{p_{1}}dy\right)^{\frac{p_{1}-1}{2}}\theta_{1}^{\frac{1}{2}}\left|\xi_{1}-\xi_{2}\right|^{\frac{p_{1}}{2}}\right.\\
	&\qquad\times		\left(1+\left|\xi_{1}\right|^{p_{1}}\theta_{1}+\left|\xi_{2}\right|^{p_{1}}\theta_{1}+\left|\xi_{1}\right|^{p_{2}}\theta_{2}+\left|\xi_{2}\right|^{p_{2}}\theta_{2}\right)^{\frac{2-p_{1}}{2}}\\	
	&\quad+ \left(\int_{Y}\chi_{2}(y)\left|P(y,\xi_{1})-P(y,\xi_{2})\right|^{p_{2}}dy\right)^{\frac{p_{1}(p_{2}-1)}{2p_{2}}}\theta_{2}^{\frac{p_{1}}{2p_{2}}}\left|\xi_{1}-\xi_{2}\right|^{\frac{p_{1}}{2}}\\
	&\qquad\times		\left(1+\left|\xi_{1}\right|^{p_{1}}\theta_{1}+\left|\xi_{2}\right|^{p_{1}}\theta_{1}+\left|\xi_{1}\right|^{p_{2}}\theta_{2}+\left|\xi_{2}\right|^{p_{2}}\theta_{2}\right)^{\frac{2-p_{1}}{2}}\\
	&\quad+ \left(\int_{Y}\chi_{1}(y)\left|P(y,\xi_{1})-P(y,\xi_{2})\right|^{p_{1}}dy\right)^{\frac{p_{2}(p_{1}-1)}{2p_{1}}}\theta_{1}^{\frac{p_{2}}{2p_{1}}}\left|\xi_{1}-\xi_{2}\right|^{\frac{p_{2}}{2}}\\
	&\qquad\times		\left(1+\left|\xi_{1}\right|^{p_{1}}\theta_{1}+\left|\xi_{2}\right|^{p_{1}}\theta_{1}+\left|\xi_{1}\right|^{p_{2}}\theta_{2}+\left|\xi_{2}\right|^{p_{2}}\theta_{2}\right)^{\frac{2-p_{2}}{2}}\\
	&\quad+ \left(\int_{Y}\chi_{2}(y)\left|P(y,\xi_{1})-P(y,\xi_{2})\right|^{p_{2}}dy\right)^{\frac{p_{2}-1}{2}}\theta_{2}^{\frac{1}{2}}\left|\xi_{1}-\xi_{2}\right|^{\frac{p_{2}}{2}}\\
	&\left.\qquad\times		\left(1+\left|\xi_{1}\right|^{p_{1}}\theta_{1}+\left|\xi_{2}\right|^{p_{1}}\theta_{1}+\left|\xi_{1}\right|^{p_{2}}\theta_{2}+\left|\xi_{2}\right|^{p_{2}}\theta_{2}\right)^{\frac{2-p_{2}}{2}}\right]
	&\intertext{Applying Young's inequality, with $r_{1}=\frac{2}{p_{1}-1}$, $r_{2}=\frac{2p_{2}}{p_{1}(p_{2}-1)}$, $r_{3}=\frac{2p_{1}}{p_{2}(p_{1}-1)}$ , and $r_{4}=\frac{2}{p_{2}-1}$.}   
	&\leq C\left[\frac{\delta^{\frac{2}{p_{1}-1}}\displaystyle\int_{Y}\chi_{1}(y)\left|P(y,\xi_{1})-P(y,\xi_{2})\right|^{p_{1}}dy}{\frac{2}{p_{1}-1}}\right.\\
	&\quad+ \frac{\delta^{\frac{-2}{3-p_{1}}}\theta_{1}^{\frac{1}{3-p_{1}}}\left|\xi_{1}-\xi_{2}\right|^{\frac{p_{1}}{3-p_{1}}}\left(1+\left|\xi_{1}\right|^{p_{1}}\theta_{1}+\left|\xi_{2}\right|^{p_{1}}\theta_{1}+\left|\xi_{1}\right|^{p_{2}}\theta_{2}+\left|\xi_{2}\right|^{p_{2}}\theta_{2}\right)^{\frac{2-p_{1}}{3-p_{1}}}}{\frac{2}{3-p_{1}}}\\	&+\frac{\delta^{\frac{2p_{2}}{p_{1}(p_{2}-1)}}\displaystyle\int_{Y}\chi_{2}(y)\left|P(y,\xi_{1})-P(y,\xi_{2})\right|^{p_{2}}dy}{\frac{2p_{2}}{p_{1}(p_{2}-1)}}\\
	&\quad+ \left(\frac{\delta^{\frac{-2p_{2}}{2p_{2}-p_{1}p_{2}+p_{1}}}\theta_{2}^{\frac{p_{1}}{2p_{2}-p_{1}p_{2}+p_{1}}}\left|\xi_{1}-\xi_{2}\right|^{\frac{p_{1}p_{2}}{2p_{2}-p_{1}p_{2}+p_{1}}}}{\frac{2p_{2}}{2p_{2}-p_{1}p_{2}+p_{1}}}\right)\\
	&\qquad\times \left(1+\left|\xi_{1}\right|^{p_{1}}\theta_{1}+\left|\xi_{2}\right|^{p_{1}}\theta_{1}+\left|\xi_{1}\right|^{p_{2}}\theta_{2}+\left|\xi_{2}\right|^{p_{2}}\theta_{2}\right)^{\frac{p_{2}(2-p_{1})}{2p_{2}-p_{1}p_{2}+p_{1}}}\\
	&\quad+ \frac{\delta^{\frac{2p_{1}}{p_{2}(p_{1}-1)}}\displaystyle\int_{Y}\chi_{1}(y)\left|P(y,\xi_{1})-P(y,\xi_{2})\right|^{p_{1}}dy}{\frac{2p_{1}}{p_{2}(p_{1}-1)}}\\
	&\quad+ \left(\frac{\delta^{\frac{-2p_{1}}{2p_{1}-p_{1}p_{2}+p_{2}}}\theta_{1}^{\frac{p_{2}}{2p_{1}-p_{1}p_{2}+p_{2}}}\left|\xi_{1}-\xi_{2}\right|^{\frac{p_{1}p_{2}}{2p_{1}-p_{1}p_{2}+p_{2}}}}{\frac{2p_{1}}{2p_{1}-p_{1}p_{2}+p_{2}}}\right)\\
	&\qquad\times	\left(1+\left|\xi_{1}\right|^{p_{1}}\theta_{1}+\left|\xi_{2}\right|^{p_{1}}\theta_{1}+\left|\xi_{1}\right|^{p_{2}}\theta_{2}+\left|\xi_{2}\right|^{p_{2}}\theta_{2}\right)^{\frac{p_{1}(2-p_{2})}{2p_{1}-p_{1}p_{2}+p_{2}}}\\	
&+\frac{\delta^{\frac{2}{p_{2}-1}}\displaystyle\int_{Y}\chi_{2}(y)\left|P(y,\xi_{1})-P(y,\xi_{2})\right|^{p_{2}}dy}{\frac{2}{p_{2}-1}}+\\
	&\quad\left. \frac{\delta^{\frac{-2}{3-p_{2}}}\theta_{2}^{\frac{1}{3-p_{2}}}\left|\xi_{1}-\xi_{2}\right|^{\frac{p_{2}}{3-p_{2}}}\left(1+\left|\xi_{1}\right|^{p_{1}}\theta_{1}+\left|\xi_{2}\right|^{p_{1}}\theta_{1}+\left|\xi_{1}\right|^{p_{2}}\theta_{2}+\left|\xi_{2}\right|^{p_{2}}\theta_{2}\right)^{\frac{2-p_{2}}{3-p_{2}}}}{\frac{2}{3-p_{2}}}\right]
\end{align*}
Straightforward algebraic manipulation delivers 
\begin{align*}
\displaystyle
	& \left(1-\frac{C\delta^{\frac{2}{p_{1}-1}}}{\frac{2}{p_{1}-1}}-\frac{C\delta^{\frac{2p_{1}}{p_{2}(p_{1}-1)}}}{\frac{2p_{1}}{p_{2}(p_{1}-1)}}\right)\int_{Y}\chi_{1}(y)\left|P(y,\xi_{1})-P(y,\xi_{2})\right|^{p_{1}}dy\\
&\quad+\left(1-\frac{C\delta^{\frac{2}{p_{2}-1}}}{\frac{2}{p_{2}-1}}-\frac{C\delta^{\frac{2p_{2}}{p_{1}(p_{2}-1)}}}{\frac{2p_{2}}{p_{1}(p_{2}-1)}}\right)\int_{Y}\chi_{2}(y)\left|P(y,\xi_{1})-P(y,\xi_{2})\right|^{p_{2}}dy\\
	&\leq C\left[ \frac{\delta^{\frac{-2}{3-p_{1}}}\theta_{1}^{\frac{1}{3-p_{1}}}\left|\xi_{1}-\xi_{2}\right|^{\frac{p_{1}}{3-p_{1}}}\left(1+\left|\xi_{1}\right|^{p_{1}}\theta_{1}+\left|\xi_{2}\right|^{p_{1}}\theta_{1}+\left|\xi_{1}\right|^{p_{2}}\theta_{2}+\left|\xi_{2}\right|^{p_{2}}\theta_{2}\right)^{\frac{2-p_{1}}{3-p_{1}}}}{\frac{2}{3-p_{1}}}\right.\\
	&\quad+  \left(\frac{\delta^{\frac{-2p_{2}}{2p_{2}-p_{1}p_{2}+p_{1}}}\theta_{2}^{\frac{p_{1}}{2p_{2}-p_{1}p_{2}+p_{1}}}\left|\xi_{1}-\xi_{2}\right|^{\frac{p_{1}p_{2}}{2p_{2}-p_{1}p_{2}+p_{1}}}}{\frac{2p_{2}}{2p_{2}-p_{1}p_{2}+p_{1}}}\right)
	&\intertext{}
	&\qquad\times	\left(1+\left|\xi_{1}\right|^{p_{1}}\theta_{1}+\left|\xi_{2}\right|^{p_{1}}\theta_{1}+\left|\xi_{1}\right|^{p_{2}}\theta_{2}+\left|\xi_{2}\right|^{p_{2}}\theta_{2}\right)^{\frac{p_{2}(2-p_{1})}{2p_{2}-p_{1}p_{2}+p_{1}}}\\
	&\quad+ \left(\frac{\delta^{\frac{-2p_{1}}{2p_{1}-p_{1}p_{2}+p_{2}}}\theta_{1}^{\frac{p_{2}}{2p_{1}-p_{1}p_{2}+p_{2}}}\left|\xi_{1}-\xi_{2}\right|^{\frac{p_{1}p_{2}}{2p_{1}-p_{1}p_{2}+p_{2}}}}{\frac{2p_{1}}{2p_{1}-p_{1}p_{2}+p_{2}}}\right)\\		
&\qquad\times\left(1+\left|\xi_{1}\right|^{p_{1}}\theta_{1}+\left|\xi_{2}\right|^{p_{1}}\theta_{1}+\left|\xi_{1}\right|^{p_{2}}\theta_{2}+\left|\xi_{2}\right|^{p_{2}}\theta_{2}\right)^{\frac{p_{1}(2-p_{2})}{2p_{1}-p_{1}p_{2}+p_{2}}}\\
	&\quad\left.+ \frac{\delta^{\frac{-2}{3-p_{2}}}\theta_{2}^{\frac{1}{3-p_{2}}}\left|\xi_{1}-\xi_{2}\right|^{\frac{p_{2}}{3-p_{2}}}\left(1+\left|\xi_{1}\right|^{p_{1}}\theta_{1}+\left|\xi_{2}\right|^{p_{1}}\theta_{1}+\left|\xi_{1}\right|^{p_{2}}\theta_{2}+\left|\xi_{2}\right|^{p_{2}}\theta_{2}\right)^{\frac{2-p_{2}}{3-p_{2}}}}{\frac{2}{3-p_{2}}}\right]
\end{align*}	
The result follows on choosing an appropriate $\delta$ and doing a change of variables. 
\subsection{Proof of Lemma~\ref{lemma3}}
Let $\Psi$ of the form (\ref{Psi}).  For every $\epsilon>0$, let us denote by 
$\displaystyle\Omega_{\epsilon}=\bigcup_{k\in I_{\epsilon}}\overline{Y_{\epsilon}^{k}};$ 
and for $j=0,1,2,...,m$, we set $$I_{\epsilon}^{j}=\left\{k\in I_{\epsilon}:Y_{\epsilon}^{k}\subseteq\Omega_{j}\right\}\text{, and } J_{\epsilon}^{j}=\left\{k\in I_{\epsilon}:Y_{\epsilon}^{k}\cap\Omega_{j}\neq\emptyset, Y_{\epsilon}^{k}\setminus\Omega_{j}\neq\emptyset\right\}.$$
Furthermore, $\displaystyle E_{\epsilon}^{j}=\bigcup_{k\in I_{\epsilon}^{j}} \overline{Y_{\epsilon}^{k}}$, 
$\displaystyle F_{\epsilon}^{j}=\bigcup_{k\in J_{\epsilon}^{j}} \overline{Y_{\epsilon}^{k}}$, and 
as $\epsilon\rightarrow0$, we have $\left|F_{\epsilon}^{j}\right|\rightarrow0$. 
		
Set $$\displaystyle \xi_{\epsilon}^{k}=\frac{1}{\left|Y_{\epsilon}^{k}\right|}\int_{Y_{\epsilon}^{k}}\varphi(y)dy.$$
For $\epsilon$ sufficiently small $\Omega_{j}$ ($j\neq0$) is contained in $\Omega_{\epsilon}$.

From (\ref{Psi}), (\ref{approximation}), using the fact that $\Omega_{j}\subset E_{\epsilon}^{j}\cup F_{\epsilon}^{j}$, 
Lemma~\ref{lemma2}, and H\"older's inequality it follows that
\begin{align}
\label{lemma3-1}
	\displaystyle
& \int_{\Omega}\chi_{1}^{\epsilon}(x)\left|P_{\epsilon}(x,M_{\epsilon}\varphi)-P_{\epsilon}(x,\Psi)\right|^{p_{1}}dx + \int_{\Omega}\chi_{2}^{\epsilon}(x)\left|P_{\epsilon}(x,M_{\epsilon}\varphi)-P_{\epsilon}(x,\Psi)\right|^{p_{2}}dx\notag\\
	& \leq
C\left[\left(\int_{\Omega}\chi_{1}^{\epsilon}(x)\left|M_{\epsilon}\varphi-\varphi\right|^{p_{1}}dx + \int_{\Omega}\chi_{1}^{\epsilon}(x)\left|\varphi-\Psi\right|^{p_{1}}dx\right)^{\frac{1}{3-p_{1}}}\right.\\
&\left.\quad\times\left(\int_{\Omega}\chi_{1}^{\epsilon}(x)\left|M_{\epsilon}\varphi-\varphi\right|^{p_{1}}dx +
\int_{\Omega}\chi_{1}^{\epsilon}(x)\left|\varphi\right|^{p_{1}}dx + \int_{\Omega}\chi_{2}^{\epsilon}(x)\left|M_{\epsilon}\varphi-\varphi\right|^{p_{2}}dx\right.\right.\notag\\
	&\left.\quad+ 
\int_{\Omega}\chi_{2}^{\epsilon}(x)\left|\varphi\right|^{p_{2}}dx + \int_{\Omega}\chi_{1}^{\epsilon}(x)\left|\Psi\right|^{p_{1}}dx + \int_{\Omega}\chi_{2}^{\epsilon}(x)\left|\Psi\right|^{p_{2}}dx+\left|\Omega\right|\right)^{\frac{2-p_{1}}{3-p_{1}}}\notag\\
&+\left(\int_{\Omega}\chi_{1}^{\epsilon}(x)\left|M_{\epsilon}\varphi-\varphi\right|^{p_{1}}dx +
\int_{\Omega}\chi_{1}^{\epsilon}(x)\left|\varphi\right|^{p_{1}}dx + \int_{\Omega}\chi_{2}^{\epsilon}(x)\left|M_{\epsilon}\varphi-\varphi\right|^{p_{2}}dx\right.\notag\\
	&\left.\quad + 
\int_{\Omega}\chi_{2}^{\epsilon}(x)\left|\varphi\right|^{p_{2}}dx + \int_{\Omega}\chi_{1}^{\epsilon}(x)\left|\Psi\right|^{p_{1}}dx + \int_{\Omega}\chi_{2}^{\epsilon}(x)\left|\Psi\right|^{p_{2}}dx+\left|\Omega\right|\right)^{\frac{p_{2}(2-p_{1})}{p_{2}-p_{1}p_{2}+p_{1}}}\notag\\
	& \quad\times
\left(\int_{\Omega}\chi_{2}^{\epsilon}(x)\left|M_{\epsilon}\varphi-\varphi\right|^{p_{2}}dx + \int_{\Omega}\chi_{2}^{\epsilon}(x)\left|\varphi-\Psi\right|^{p_{2}}dx\right)^{\frac{p_{1}}{p_{2}-p_{1}p_{2}+p_{1}}}\notag\\
	&+\left( \int_{\Omega}\chi_{1}^{\epsilon}(x)\left|M_{\epsilon}\varphi-\varphi\right|^{p_{1}}dx +
\int_{\Omega}\chi_{1}^{\epsilon}(x)\left|\varphi\right|^{p_{1}}dx + \int_{\Omega}\chi_{2}^{\epsilon}(x)\left|M_{\epsilon}\varphi-\varphi\right|^{p_{2}}dx\right.\notag
&\intertext{}
	&\left.\quad + 
\int_{\Omega}\chi_{2}^{\epsilon}(x)\left|\varphi\right|^{p_{2}}dx + \int_{\Omega}\chi_{1}^{\epsilon}(x)\left|\Psi\right|^{p_{1}}dx + \int_{\Omega}\chi_{2}^{\epsilon}(x)\left|\Psi\right|^{p_{2}}dx+\left|\Omega\right|\right)^{\frac{p_{1}(2-p_{2})}{p_{1}-p_{1}p_{2}+p_{2}}}\notag\\
	& \quad\times
\left(\int_{\Omega}\chi_{1}^{\epsilon}(x)\left|M_{\epsilon}\varphi-\varphi\right|^{p_{1}}dx + \int_{\Omega}\chi_{1}^{\epsilon}(x)\left|\varphi-\Psi\right|^{p_{1}}dx\right)^{\frac{p_{2}}{p_{1}-p_{1}p_{2}+p_{2}}}\notag\\
	&+\left( \int_{\Omega}\chi_{1}^{\epsilon}(x)\left|M_{\epsilon}\varphi-\varphi\right|^{p_{1}}dx +
\int_{\Omega}\chi_{1}^{\epsilon}(x)\left|\varphi\right|^{p_{1}}dx + \int_{\Omega}\chi_{2}^{\epsilon}(x)\left|M_{\epsilon}\varphi-\varphi\right|^{p_{2}}dx\right.\notag\\
	&\quad\left.+ 
\int_{\Omega}\chi_{2}^{\epsilon}(x)\left|\varphi\right|^{p_{2}}dx + \int_{\Omega}\chi_{1}^{\epsilon}(x)\left|\Psi\right|^{p_{1}}dx + \int_{\Omega}\chi_{2}^{\epsilon}(x)\left|\Psi\right|^{p_{2}}dx+\left|\Omega\right|\right)^{\frac{2-p_{2}}{3-p_{2}}}\notag\\
& \left.\quad\times
\left(\int_{\Omega}\chi_{2}^{\epsilon}(x)\left|M_{\epsilon}\varphi-\varphi\right|^{p_{2}}dx + \int_{\Omega}\chi_{2}^{\epsilon}(x)\left|\varphi-\Psi\right|^{p_{2}}dx\right)^{\frac{1}{3-p_{2}}}\right]\notag\\
&+C\sum_{j=0}^{m}\left[\left(\int_{F_{\epsilon}^{j}}\theta_{1}\left|\sum_{k\in J_{\epsilon}^{j}}\chi_{Y_{\epsilon}^{k}}\xi_{\epsilon}^{k}-\eta_{j}\right|^{p_{1}}dx\right)^{\frac{1}{3-p_{1}}}\left(\left|F_{\epsilon}^{j}\right| + \int_{F_{\epsilon}^{j}}\left|M_{\epsilon}\varphi\right|^{p_{1}}\theta_{1}dx \right.\right.\notag\\
	&\quad\left.+ \int_{F_{\epsilon}^{j}}\left|M_{\epsilon}\varphi\right|^{p_{2}}\theta_{2}dx+ 
\left|\eta_{j}\right|^{p_{1}}\theta_{1}\left|F_{\epsilon}^{j}\right| + \left|\eta_{j}\right|^{p_{2}}\theta_{2}\left|F_{\epsilon}^{j}\right|\right)^{\frac{2-p_{1}}{3-p_{1}}}\notag\\
	&+\left(\int_{F_{\epsilon}^{j}}\theta_{2}\left|\sum_{k\in J_{\epsilon}^{j}}\chi_{Y_{\epsilon}^{k}}\xi_{\epsilon}^{k}-\eta_{j}\right|^{p_{2}}dx\right)^{\frac{p_{1}}{2p_{2}-p_{1}p_{2}+p_{1}}} 
\left(\left|F_{\epsilon}^{j}\right| + \int_{F_{\epsilon}^{j}}\left|M_{\epsilon}\varphi\right|^{p_{1}}\theta_{1}dx\right.\notag\\
	&\quad\left. + \int_{F_{\epsilon}^{j}}\left|M_{\epsilon}\varphi\right|^{p_{2}}\theta_{2}dx+ 
\left|\eta_{j}\right|^{p_{1}}\theta_{1}\left|F_{\epsilon}^{j}\right| + \left|\eta_{j}\right|^{p_{2}}\theta_{2}\left|F_{\epsilon}^{j}\right|\right)^{\frac{p_{2}(2-p_{1})}{2p_{2}-p_{1}p_{2}+p_{1}}} \notag\\
	&+\left(\int_{F_{\epsilon}^{j}}\theta_{1}\left|\sum_{k\in J_{\epsilon}^{j}}\chi_{Y_{\epsilon}^{k}}\xi_{\epsilon}^{k}-\eta_{j}\right|^{p_{1}}dx\right)^{\frac{p_{2}}{2p_{1}-p_{1}p_{2}+p_{2}}} 
\left(\left|F_{\epsilon}^{j}\right| + \int_{F_{\epsilon}^{j}}\left|M_{\epsilon}\varphi\right|^{p_{1}}\theta_{1}dx \right.\notag\\
	&\quad\left.+ \int_{F_{\epsilon}^{j}}\left|M_{\epsilon}\varphi\right|^{p_{2}}\theta_{2}dx+ 
\left|\eta_{j}\right|^{p_{1}}\theta_{1}\left|F_{\epsilon}^{j}\right| + \left|\eta_{j}\right|^{p_{2}}\theta_{2}\left|F_{\epsilon}^{j}\right|\right)^{\frac{p_{1}(2-p_{2})}{2p_{1}-p_{1}p_{2}+p_{2}}}\notag\\
	&+\left(\int_{F_{\epsilon}^{j}}\theta_{2}\left|\sum_{k\in J_{\epsilon}^{j}}\chi_{Y_{\epsilon}^{k}}\xi_{\epsilon}^{k}-\eta_{j}\right|^{p_{2}}dx\right)^{\frac{1}{3-p_{2}}} 
\left(\left|F_{\epsilon}^{j}\right| + \int_{F_{\epsilon}^{j}}\left|M_{\epsilon}\varphi\right|^{p_{1}}\theta_{1}dx \right.\notag\\	
	&\quad\left.\left.+ \int_{F_{\epsilon}^{j}}\left|M_{\epsilon}\varphi\right|^{p_{2}}\theta_{2}dx+ 
\left|\eta_{j}\right|^{p_{1}}\theta_{1}\left|F_{\epsilon}^{j}\right| + \left|\eta_{j}\right|^{p_{2}}\theta_{2}\left|F_{\epsilon}^{j}\right|\right)^{\frac{2-p_{2}}{3-p_{2}}}\right]\notag
\end{align}

Since $\left|\partial\Omega_{j}\right|=0$ for $j\neq0$, we have that $\left|F_{\epsilon}^{j}\right|\rightarrow0$ 
as $\epsilon\rightarrow0$, for every $j=0,1,2,...,m$.
	
By Property~(1) of $M_{\epsilon}$ mentioned in Section~\ref{CorrectorSection}, we have $$\displaystyle\int_{\Omega}\chi_{i}^{\epsilon}(x)\left|M_{\epsilon}\varphi(x)-\varphi(x)\right|^{p_{i}}dx\rightarrow0, \text{ as $\epsilon\rightarrow0$, for $i=1,2$.}$$ 
	
Therefore, taking $\limsup$ as $\epsilon\rightarrow0$ in (\ref{lemma3-1}), we obtain (\ref{lemma3formula}).

\subsection{Proof of Lemma~\ref{proofaprioribound}}
Evaluating $u_{\epsilon}$ in the weak formulation for (\ref{Dirichlet}), applying H\"older's inequality, and 
since $f\in W^{-1,q_{2}}(\Omega)$, we obtain 
\begin{align}
	\label{aprioriboundproof1}
	\displaystyle	
	&\int_{\Omega}(A_{\epsilon}(x,\nabla u_{\epsilon}),\nabla u_{\epsilon})dx=\sigma_{1}\int_{\Omega}\chi_{1}^{\epsilon}(x)\left|\nabla u_{\epsilon}\right|^{p_{1}}dx+\sigma_{2}\int_{\Omega}\chi_{2}^{\epsilon}(x)\left|\nabla u_{\epsilon}\right|^{p_{2}}dx\\
	&\quad =\left\langle f, u_{\epsilon}\right\rangle\leq C\left[\left(\int_{\Omega}\chi_{1}^{\epsilon}(x)\left|\nabla u_{\epsilon}\right|^{p_{1}}dx\right)^{\frac{1}{p_{1}}}+\left(\int_{\Omega}\chi_{2}^{\epsilon}(x)\left|\nabla u_{\epsilon}\right|^{p_{2}}dx\right)^{\frac{1}{p_{2}}}\right]\notag
\end{align}
Applying Young's inequality to the last term in (\ref{aprioriboundproof1}), we obtain	
\begin{align}
	\label{aprioriboundproof2}
	\displaystyle
	& \sigma_{1}\int_{\Omega}\chi_{1}^{\epsilon}(x)\left|\nabla u_{\epsilon}\right|^{p_{1}}dx+\sigma_{2}\int_{\Omega}\chi_{2}^{\epsilon}(x)\left|\nabla u_{\epsilon}\right|^{p_{2}}dx\\
	& \quad \leq C\left[\frac{\delta^{p_{1}}}{p_{1}}\int_{\Omega}\chi_{1}^{\epsilon}(x)\left|\nabla u_{\epsilon}\right|^{p_{1}}dx+\frac{\delta^{-q_{2}}}{q_{2}}+\frac{\delta^{p_{2}}}{p_{2}}\int_{\Omega}\chi_{2}^{\epsilon}(x)\left|\nabla u_{\epsilon}\right|^{p_{2}}dx+\frac{\delta^{-q_{1}}}{q_{1}}\right]\notag
\end{align}	
By rearranging the terms in (\ref{aprioriboundproof2}), one gets
\begin{align*}
	\displaystyle
&\left(\sigma_{1}-C\frac{\delta^{p_{1}}}{p_{1}}\right)\int_{\Omega}\chi_{1}^{\epsilon}(x)\left|\nabla u_{\epsilon}\right|^{p_{1}}dx+\left(\sigma_{2}-C\frac{\delta^{p_{2}}}{p_{2}}\right)\int_{\Omega}\chi_{2}^{\epsilon}(x)\left|\nabla u_{\epsilon}\right|^{p_{2}}dx\\
	&\quad\leq \frac{\delta^{-q_{2}}}{q_{2}}+\frac{\delta^{-q_{1}}}{q_{1}}.
\end{align*}
Therefore, by choosing $\delta$ small enough so that $\min\left\{\sigma_{1}-C\frac{\delta^{p_{1}}}{p_{1}},\sigma_{2}-C\frac{\delta^{p_{2}}}{p_{2}}\right\}$ is positive, one obtains $$\int_{\Omega}\chi_{1}^{\epsilon}(x)\left|\nabla u_{\epsilon}(x)\right|^{p_{1}}dx+\int_{\Omega}\chi_{2}^{\epsilon}(x)\left|\nabla u_{\epsilon}(x)\right|^{p_{2}}dx\leq C.$$

\subsection{Proof of Lemma~\ref{uniform boundedness of p at M}}
Using (\ref{approximation}), we have
\begin{align*}
	\displaystyle 
	& 
\int_{\Omega}\chi_{1}^{\epsilon}(x)\left|P_{\epsilon}(x,M_{\epsilon}\nabla u(x))\right|^{p_{1}}dx + \int_{\Omega}\chi_{2}^{\epsilon}(x)\left|P_{\epsilon}(x,M_{\epsilon}\nabla u(x))\right|^{p_{2}}dx\\
	&\quad = \sum_{k\in\textbf{I}_{\epsilon}}\left[\int_{Y_{\epsilon}^{k}}\chi_{1}^{\epsilon}(x)\left|P_{\epsilon}(x,\xi_{\epsilon}^{k})\right|^{p_{1}}dx + \int_{Y_{\epsilon}^{k}}\chi_{2}^{\epsilon}(x)\left|P_{\epsilon}(x,\xi_{\epsilon}^{k})\right|^{p_{2}}dx\right]\\
	&\quad \leq   C\sum_{k\in\textbf{I}_{\epsilon}}\left(\left|Y_{\epsilon}^{k}\right|+\left|\xi_{\epsilon}^{k}\right|^{p_{1}}\theta_{1}\left|Y_{\epsilon}^{k}\right|+\left|\xi_{\epsilon}^{k}\right|^{p_{2}}\theta_{2}\left|Y_{\epsilon}^{k}\right|\right)\\
	&\quad \leq C\left(\left|\Omega\right|+\left\|\nabla u\right\|^{p_{1}}_{\textbf{L}^{p_{1}}(\Omega)}+\left\|\nabla u\right\|^{p_{2}}_{\textbf{L}^{p_{2}}(\Omega)}\right)<\infty,
	\end{align*}
where the last three inequalities follow from Lemma~\ref{lemma1}, Jensen's inequality, and Theorem~\ref{regularity}.

\subsection{Proof of Lemma~\ref{monconb}}
We prove properties (\ref{Monb}) and (\ref{Conb}) of the homogenized operator $b$.  Property (\ref{Conb2}), which occurs in the case when $1<p_{1}\leq2\leq p_{2}$, follows in a similar way.  
\subsubsection{Proof of (\ref{Monb})}
Using (\ref{cell}) and (\ref{MonA}), we have
\begin{align*}
\displaystyle 
	&\left(b(\xi_{2}) - b(\xi_{1}),\xi_{2}-\xi_{1}\right)\\
&\quad=\int_{Y}\left(A(y,P(y,\xi_{2}))-A(y,P(y,\xi_{1})),P(y,\xi_{2})-P(y,\xi_{1})\right)dy\\	
	&\quad\geq C\int_{Y}\left|P(y,\xi_{1})-P(y,\xi_{2})\right|^{\beta(y)}\left(\left|P(y,\xi_{1})\right|+\left|P(y,\xi_{2})\right|\right)^{p(y)-\beta(y)}dy\geq 0. 
\end{align*}
\subsubsection{Proof of (\ref{Conb})}
Note that 
\begin{align}
\label{conb1}
\displaystyle
&\int_{Y}\chi_{1}(y)\left|P(y,\xi_{1})-P(y,\xi_{2})\right|^{2}\left(\left|P(y,\xi_{1})\right|+\left|P(y,\xi_{2})\right|\right)^{p_{1}-2}dy\notag\\
	&\quad +\int_{Y}\chi_{2}(y)\left|P(y,\xi_{1})-P(y,\xi_{2})\right|^{2}\left(\left|P(y,\xi_{1})\right|+\left|P(y,\xi_{2})\right|\right)^{p_{2}-2}dy\\
	&\qquad \leq C\int_{Y}\left(A(y,P(y,\xi_{2}))-A(y,P(y,\xi_{1})),P(y,\xi_{2})-P(y,\xi_{1})\right)dy\notag\\
	&\qquad = C\int_{Y}\left(A(y,P(y,\xi_{2}))-A(y,P(y,\xi_{1})),\xi_{2}-\xi_{1}\right)dy\notag\\
	&\qquad =\left(b(\xi_{2})-b(\xi_{1}),\xi_{2}-\xi_{1}\right)\leq \left|b(\xi_{2})-b(\xi_{1})\right|\left|\xi_{2}-\xi_{1}\right|\notag
\end{align}
by (\ref{MonA}), (\ref{cell}), and (\ref{b}).
Also, using (\ref{ConA}), we obtain 
\begin{align*}
	\displaystyle 
	& \left|b(\xi_{1}) - b(\xi_{2})\right|\\
	&\leq\int_{Y}\left|A(y,P(y,\xi_{1}))-A(y,P(y,\xi_{2}))\right|dy \\
	&\leq C\left[\int_{Y}\chi_{1}(y)\left|P(y,\xi_{1})-P(y,\xi_{2})\right|^{p_{1}-1}dy + \int_{Y}\chi_{2}(y)\left|P(y,\xi_{1})-P(y,\xi_{2})\right|^{p_{2}-1}dy \right]\\
	& =C\left[\int_{Y}\frac{\chi_{1}(y)\left|P(y,\xi_{1})-P(y,\xi_{2})\right|^{p_{1}-1}\left(\left|P(x,\xi_{1})\right|+\left|P(x,\xi_{2})\right|\right)^{\frac{(2-p_{1})(p_{1}-1)}{2}}}{\left(\left|P(x,\xi_{1})\right|+\left|P(x,\xi_{2})\right|\right)^{\frac{(2-p_{1})(p_{1}-1)}{2}}}dy\right.\\
	&\quad\left.+ \int_{Y}\frac{\chi_{2}(y)\left|P(y,\xi_{1})-P(y,\xi_{2})\right|^{p_{2}-1}\left(\left|P(x,\xi_{1})\right|+\left|P(x,\xi_{2})\right|\right)^{\frac{(2-p_{2})(p_{2}-1)}{2}}}{\left(\left|P(x,\xi_{1})\right|+\left|P(x,\xi_{2})\right|\right)^{\frac{(2-p_{2})(p_{2}-1)}{2}}}dy \right]
	&\intertext{Using H\"older's inequality with $r_{i}=\frac{2}{p_{i}-1}>1$ and $\frac{1}{s_{i}}=1-\frac{1}{r_{i}}=\frac{3-p_{i}}{2}$, for $i=1,2$}
	&\leq  C\left[\left(\int_{Y}\frac{\chi_{1}(y)\left|P(y,\xi_{1})-P(y,\xi_{2})\right|^{2}}{\left(\left|P(x,\xi_{1})\right|+\left|P(x,\xi_{2})\right|\right)^{2-p_{1}}}dy\right)^{\frac{p_{1}-1}{2}}\right.\\
	&\qquad\times \left(\int_{Y}\chi_{1}(y)\left(\left|P(x,\xi_{1})\right|+\left|P(x,\xi_{2})\right|\right)^{\frac{(2-p_{1})(p_{1}-1)}{3-p_{1}}}dy\right)^{\frac{3-p_{1}}{2}}\\
	&\quad+ \left(\int_{Y}\frac{\chi_{2}(y)\left|P(y,\xi_{1})-P(y,\xi_{2})\right|^{2}}{\left(\left|P(x,\xi_{1})\right|+\left|P(x,\xi_{2})\right|\right)^{2-p_{2}}}dy\right)^{\frac{p_{2}-1}{2}}\\
	&\qquad\left.\times \left(\int_{Y}\chi_{2}(y)\left(\left|P(x,\xi_{1})\right|+\left|P(x,\xi_{2})\right|\right)^{\frac{(2-p_{2})(p_{2}-1)}{3-p_{2}}}dy\right)^{\frac{3-p_{2}}{2}} \right]
	&\intertext{Using H\"older's inequality with $r_{i}=\frac{p_{i}(3-p_{i})}{(p_{i}-1)(2-p_{i})}>1$, for $i=1,2$}
	&\leq C\left[\left(\int_{Y}\frac{\chi_{1}(y)\left|P(y,\xi_{1})-P(y,\xi_{2})\right|^{2}}{\left(\left|P(x,\xi_{1})\right|+\left|P(x,\xi_{2})\right|\right)^{2-p_{1}}}dy\right)^{\frac{p_{1}-1}{2}}\right.\\
&\qquad\times\left(\int_{Y}\chi_{1}(y)\left(\left|P(x,\xi_{1})\right|+\left|P(x,\xi_{2})\right|\right)^{p_{1}}dy\right)^{\frac{(2-p_{1})(p_{1}-1)}{2p_{1}}}\\
	&\quad+ \left(\int_{Y}\frac{\chi_{2}(y)\left|P(y,\xi_{1})-P(y,\xi_{2})\right|^{2}}{\left(\left|P(x,\xi_{1})\right|+\left|P(x,\xi_{2})\right|\right)^{2-p_{2}}}dy\right)^{\frac{p_{2}-1}{2}}\\
	&\qquad\left.\times \left(\int_{Y}\chi_{2}(y)\left(\left|P(x,\xi_{1})\right|+\left|P(x,\xi_{2})\right|\right)^{p_{2}}dy\right)^{\frac{(2-p_{2})(p_{2}-1)}{2p_{2}}}\right]
	&\intertext{By (\ref{conb1}) and Lemma (\ref{lemma1})}
	&\leq C\sum_{i=1}^2\left[\left|b(\xi_{2})-b(\xi_{1})\right|^{\frac{p_{i}-1}{2}}\left|\xi_{2}-\xi_{1}\right|^{\frac{p_{i}-1}{2}}\right.\\
	&\quad\left.\times	\left(1+\theta_{1}\left|\xi_{1}\right|^{p_{1}}+\theta_{2}\left|\xi_{1}\right|^{p_{2}}+\theta_{1}\left|\xi_{2}\right|^{p_{1}}+\theta_{2}\left|\xi_{2}\right|^{p_{2}}\right)^{\frac{(2-p_{i})(p_{i}-1)}{2p_{i}}}\right]
	&\intertext{By Young's Inequality with $r_{i}=\frac{2}{p_{i}-1}>1$ with $i=1,2$}
	&\leq C\sum_{i=1}^{2}\left[\frac{\delta^{\frac{2}{p_{i}-1}}\left|b(\xi_{2})-b(\xi_{1})\right|}{\frac{2}{p_{i}-1}}+\right.\\
&\quad\left.\frac{\delta^{\frac{-2}{3-p_{i}}}\left|\xi_{2}-\xi_{1}\right|^{\frac{p_{i}-1}{3-p_{i}}}\left(1+\theta_{1}\left|\xi_{1}\right|^{p_{1}}+\theta_{2}\left|\xi_{1}\right|^{p_{2}}+\theta_{1}\left|\xi_{2}\right|^{p_{1}}+\theta_{2}\left|\xi_{2}\right|^{p_{2}}\right)^{\frac{(2-p_{i})(p_{i}-1)}{p_{i}(3-p_{i})}}}{\frac{2}{3-p_{i}}}\right]			
\end{align*}
Rearranging the terms, and taking $\delta$ small enough we obtain (\ref{Conb})

\subsection{Proof of Lemma~\ref{unifboundDP}}
Using H\"older's inequality, (\ref{ConA}), and (\ref{aprioribound}), we obtain
\begin{align*}	
	\displaystyle 
&\int_{\Omega_{j}}\left|\left(A_{\epsilon}\left(x,P_{\epsilon}\left(x,\eta_{j}\right)\right),\nabla u_{\epsilon}(x)\right)\right|dx\leq \int_{\Omega_{j}}\left|A_{\epsilon}\left(x,P_{\epsilon}\left(x,\eta_{j}\right)\right)\right|\left|\nabla u_{\epsilon}(x)\right|dx\\
	&\quad \leq C\left[\left(\int_{\Omega_{j}}\chi_{1}^{\epsilon}(x)\left|P_{\epsilon}\left(x,\eta_{j}\right)\right|^{p_{1}}dx\right)^{\frac{1}{q_{2}}}+\left(\int_{\Omega_{j}}\chi_{2}^{\epsilon}(x)\left|P_{\epsilon}\left(x,\eta_{j}\right)\right|^{p_{2}}dx\right)^{\frac{1}{q_{1}}}\right]\\
	&\quad\leq C \text{, where $C$ does not depend on $\epsilon$}.	
\end{align*}
The proof of the uniform boundedness of $\displaystyle \int_{\Omega_{j}}\left|\left(A_{\epsilon}\left(x,\nabla u_{\epsilon}(x)\right),P_{\epsilon}\left(x,\eta_{j}\right)\right)\right|dx$ follows in the same manner.

\subsection{Proof of Lemma~\ref{dunfordpettis}}
We prove the first statement of the lemma, the second statement follows in a similar way.  The lemma follows from the Dunford-Pettis~theorem (see \cite{Dacorogna1989}).  To apply this theorem, the following conditions are necessary:
\begin{itemize}
	\item $\displaystyle \int_{\Omega_{j}}\left|\left(A_{\epsilon}\left(x,P_{\epsilon}\left(x,\eta_{j}\right)\right),\nabla u_{\epsilon}(x)\right)\right|dx$ 
	is uniformly bounded with respect to $\epsilon$, which was proved in Lemma~\ref{unifboundDP}.
	\item $\left(A_{\epsilon}\left(\cdot,P_{\epsilon}\left(\cdot,\eta_{j}\right)\right),\nabla u_{\epsilon}(\cdot)\right)$ 	is equiintegrable for all $j=0,..,m$
	\begin{proof}
	By Theorem~1.5 in \cite{Dacorogna1989}, we have that $\chi_{1}^{\epsilon}(\cdot)\left|A_{\epsilon}\left(\cdot,P_{\epsilon}\left(\cdot,\eta_{j}\right)\right)\right|^{q_{2}}$ and $\chi_{2}^{\epsilon}(\cdot)\left|A_{\epsilon}\left(\cdot,P_{\epsilon}\left(\cdot,\eta_{j}\right)\right)\right|^{q_{1}}$ are equiintegrable.
	
	By (\ref{aprioribound}), for any $E\subset\Omega$, we have $$\max_{i=1,2}\left\{\sup_{\epsilon>0}\left\{\left(\int_{E}\chi_{i}^{\epsilon}(x)\left|\nabla u_{\epsilon}(x)\right|^{p_{i}}dx\right)^{\frac{1}{p_{i}}}\right\}\right\}\leq C.$$
	Let $\lambda>0$ arbitrary and choose $\lambda_{1}>0$ and $\lambda_{2}>0$ such that $\lambda_{1}^{\frac{1}{q_{2}}}+\lambda_{2}^{\frac{1}{q_{1}}}<\lambda/C$.
	
	For $\lambda_{1}$ and $\lambda_{2}$, there exist $\omega(\lambda_{1})>0$ and  $\omega(\lambda_{2})>0$ such that for every $E\subset\Omega$ with $\left|E\right|<\min\left\{\omega(\lambda_{1}),\omega(\lambda_{2})\right\}$, we have $$\int_{E}\chi_{1}^{\epsilon}(x)\left|A_{\epsilon}\left(x,P_{\epsilon}\left(x,\eta_{j}\right)\right)\right|^{q_{2}}dx<\lambda_{1} \text{\hspace{2mm} and \hspace{2mm}} \int_{E}\chi_{2}^{\epsilon}(x)\left|A_{\epsilon}\left(x,P_{\epsilon}\left(x,\eta_{j}\right)\right)\right|^{q_{1}}dx<\lambda_{2}$$
	Take $\omega=\omega(\lambda)=\min\left\{\omega(\lambda_{1}),\omega(\lambda_{2})\right\}$.  Then, for all $E\subset\Omega$ with $\left|E\right|<\omega$, we have
\begin{align*}
	\displaystyle 
&\int_{E}\left|\left(A_{\epsilon}\left(x,P_{\epsilon}\left(x,\eta_{j}\right)\right),\nabla u_{\epsilon}\right)\right|dx\leq \int_{E}\left|A_{\epsilon}\left(x,p_{\epsilon}\left(x,\eta_{j}\right)\right)\right|\left|\nabla u_{\epsilon}\right|dx\\
	&\quad\leq \left(\int_{E}\chi_{1}^{\epsilon}(x)\left|A_{\epsilon}\left(x,P_{\epsilon}\left(x,\eta_{j}\right)\right)\right|^{q_{2}}dx\right)^{\frac{1}{q_{2}}}\left(\int_{E}\chi_{1}^{\epsilon}(x)\left|\nabla u_{\epsilon}\right|^{p_{1}}dx\right)^{\frac{1}{p_{1}}}\\
	&\qquad+ \left(\int_{E}\chi_{2}^{\epsilon}(x)\left|A_{\epsilon}\left(x,P_{\epsilon}\left(x,\eta_{j}\right)\right)\right|^{q_{1}}dx\right)^{\frac{1}{q_{1}}}\left(\int_{E}\chi_{2}^{\epsilon}(x)\left|\nabla u_{\epsilon}\right|^{p_{2}}dx\right)^{\frac{1}{p_{2}}}\\
	&\quad\leq C(\lambda_{1}^{1/q_{2}}+\lambda_{2}^{1/q_{1}}) < \lambda,
\end{align*}
for every $\lambda>0$, and so $\left(A_{\epsilon}\left(\cdot,P_{\epsilon}\left(\cdot,\eta_{j}\right)\right),\nabla u_{\epsilon}\right)$ is equiintegrable. 	
\end{proof}  
\end{itemize}

\subsection{Case $1<p_{1}\leq 2\leq p_{2}$: Proof of Theorem~\ref{corrector} in Section~\ref{main}.}
\label{appendixcorrector}
The proof of Theorem~\ref{corrector} for the case $1<p_{1}\leq 2\leq p_{2}$ is very similar to the one presented in Section~\ref{main}.  Here we will indicate only the main differences in the different parts of the proof.  

We have by (\ref{MonA}), Lemma~\ref{proofaprioribound}, and Lemma~{\ref{uniform boundedness of p at M}} that
\begin{align*}
	\displaystyle	
	& 
\int_{\Omega}\left[\chi_{1}^{\epsilon}(x)\left|P_{\epsilon}\left(x,M_{\epsilon}\nabla u(x)\right)-\nabla u_{\epsilon}(x)\right|^{p_{1}} + \chi_{2}^{\epsilon}(x)\left|P_{\epsilon}\left(x,M_{\epsilon}\nabla u(x)\right)-\nabla u_{\epsilon}(x)\right|^{p_{2}}\right]dx\\
	&\leq C\left[\left(\left|\Omega\right|+\int_{\Omega}\chi_{1}^{\epsilon}(x)\left|P_{\epsilon}(x,M_{\epsilon}\nabla u(x))\right|^{p_{1}}dx+\int_{\Omega}\chi_{1}^{\epsilon}(x)\left|\nabla u_{\epsilon}(x)\right|^{p_{1}}dx\right)^{\frac{2-p_{1}}{2}}\right.\\
&\times\left(\int_{\Omega}\chi_{1}^{\epsilon}(x)\left(A_{\epsilon}\left(x,P_{\epsilon}\left(x,M_{\epsilon} \nabla u(x)\right)\right)-A_{\epsilon}\left(x,\nabla u_{\epsilon}(x)\right),P_{\epsilon}\left(x,M_{\epsilon}\nabla u(x)\right)-\nabla u_{\epsilon}(x)\right)dx\right)^{\frac{p_{1}}{2}}\\
	&\left.+ \int_{\Omega}\chi_{2}^{\epsilon}(x)\left(A_{\epsilon}\left(x,P_{\epsilon}\left(x,M_{\epsilon} \nabla u(x)\right)\right)-A_{\epsilon}\left(x,\nabla u_{\epsilon}(x)\right),P_{\epsilon}\left(x,M_{\epsilon}\nabla u(x)\right)-\nabla u_{\epsilon}(x)\right)dx\right]
	&\intertext{}
	&\leq C\left[\left(\int_{\Omega}\left(A_{\epsilon}\left(x,P_{\epsilon}\left(x,M_{\epsilon} \nabla u(x)\right)\right)-A_{\epsilon}\left(x,\nabla u_{\epsilon}(x)\right),P_{\epsilon}\left(x,M_{\epsilon}\nabla u(x)\right)-\nabla u_{\epsilon}(x)\right)dx\right)^{\frac{p_{1}}{2}}\right.\\
	&\left.+ \int_{\Omega}\left(A_{\epsilon}\left(x,P_{\epsilon}\left(x,M_{\epsilon} \nabla u(x)\right)\right)-A_{\epsilon}\left(x,\nabla u_{\epsilon}(x)\right),P_{\epsilon}\left(x,M_{\epsilon}\nabla u(x)\right)-\nabla u_{\epsilon}(x)\right)dx\right]	
\end{align*}
Therefore to prove Theorem~\ref{corrector} in this case, we also need to show that 
\begin{align*}
	\displaystyle
	& \int_{\Omega}\left(A_{\epsilon}\left(x,P_{\epsilon}\left(x,M_{\epsilon}\nabla u(x)\right)\right)-A_{\epsilon}\left(x,\nabla u_{\epsilon}(x)\right),P_{\epsilon}\left(x,M_{\epsilon}\nabla u(x)\right)-\nabla u_{\epsilon}(x)\right)dx\\
	&=\int_{\Omega}\left(A_{\epsilon}\left(x,P_{\epsilon}\left(x,M_{\epsilon}\nabla u\right)\right),P_{\epsilon}\left(x,M_{\epsilon}\nabla u\right)\right)dx-\int_{\Omega}\left(A_{\epsilon}\left(x,P_{\epsilon}\left(x,M_{\epsilon}\nabla u\right)\right),\nabla u_{\epsilon}\right)dx\\
	&\quad-\int_{\Omega}\left(A_{\epsilon}\left(x,\nabla u_{\epsilon}\right),P_{\epsilon}\left(x,M_{\epsilon}\nabla u\right)\right)dx+\int_{\Omega}\left(A_{\epsilon}\left(x,\nabla u_{\epsilon}\right),\nabla u_{\epsilon}\right)dx
\end{align*}
goes to 0 as $\epsilon\rightarrow0$ and this is done with the same four steps as in Section~\ref{main}.

In Step~1, by (\ref{Conb2}), H\"older's inequality, Theorem~\ref{regularity}, and Jensen's inequality, we obtain 
\begin{align*}
	\displaystyle
	&\int_{\Omega}\left|b(M_{\epsilon}\nabla u(x))-b(\nabla u(x))\right|^{q_{1}}dx\\
	&\quad\leq C\left[\left(\int_{\Omega}\left|M_{\epsilon}\nabla u-\nabla u\right|^{p_{2}}dx\right)^{\frac{p_{1}-1}{(p_{2}-1)(3-p_{1})}}+ \left(\int_{\Omega}\left|M_{\epsilon}\nabla u-\nabla u\right|^{p_{2}}dx\right)^{\frac{1}{(p_{2}-1)^2}}\right]			
\end{align*}

In Step~2, by (\ref{ConA}), H\"older's inequality, and (\ref{aprioribound}) we obtain (instead of (\ref{step2-1}))
\begin{align}
	\label{step2-1b}
	\displaystyle 
&\left|\int_{\Omega}\left(A_{\epsilon}\left(x,P_{\epsilon}\left(x,M_{\epsilon}\nabla u(x)\right)\right)-A_{\epsilon}\left(x,P_{\epsilon}\left(x,\Psi(x) \right)\right),\nabla u_{\epsilon}(x)\right)dx\right|\\
	&\quad\leq C\left[\left(\int_{\Omega}\chi_{1}^{\epsilon}(x)\left|P_{\epsilon}\left(x,M_{\epsilon}\nabla u\right)-P_{\epsilon}\left(x,\Psi \right)\right|^{p_{1}}dx\right)^{\frac{p_{1}-1}{p_{1}}}\right.\notag\\
&\qquad\left.+\left(\int_{\Omega}\chi_{2}^{\epsilon}(x)\left|P_{\epsilon}\left(x,M_{\epsilon}\nabla u\right)-P_{\epsilon}\left(x,\Psi \right)\right|^{p_{2}}dx\right)^{\frac{1}{p_{2}}}\right]\notag
\end{align}

Applying Lemma~\ref{lemma3} and (\ref{approximation with simple function of Du}) to (\ref{step2-1b}), we discover that	
\begin{align}
	\label{step2-3b}
	\displaystyle & \limsup_{\epsilon\rightarrow0}\left|\int_{\Omega}\left(A_{\epsilon}\left(x,P_{\epsilon}\left(x,M_{\epsilon}\nabla u(x)\right)\right)-A_{\epsilon}\left(x,P_{\epsilon}\left(x,\Psi(x)\right)\right),\nabla u_{\epsilon}(x)\right)dx\right|\notag\\
	&\quad \leq C\left[\left(\delta^{\frac{p_{1}}{3-p_{1}}}+\delta^{\frac{p_{1}p_{2}}{2p_{2}-p_{1}}}+\delta^{p_{1}}+\delta^{\frac{p_{2}}{p_{2}-1}}\right)^\frac{p_{1}-1}{p_{1}}\right.\\
	&\qquad\left. +\left(\delta^{\frac{p_{1}}{3-p_{1}}}+\delta^{\frac{p_{1}p_{2}}{2p_{2}-p_{1}}}+\delta^{p_{1}}+\delta^{\frac{p_{2}}{p_{2}-1}}\right)^\frac{1}{p_{2}}\right],\notag
\end{align}
where $C$ is independent of $\delta$.  Since $\delta$ is arbitrary we conclude that the limit on the left hand side of (\ref{step2-3b}) is equal to $0$.

Finally, using the continuity of $b$ (\ref{Conb2}) in Lemma~\ref{monconb}, Theorem~\ref{regularity}, and H\"older's inequality, we obtain $$\left|\int_{\Omega}\left(b(\nabla u(x))-b(\Psi(x)),\nabla u(x)\right)dx\right|\leq C\left[\delta^{\frac{p_{2}(p_{1}-1)}{(p_{2}-1)(3-p_{1})}}+\delta^{\frac{p_{2}}{(p_{2}-1)^2}}\right]^{\frac{1}{q_{1}}},$$
where $C$ does not depend on $\delta$. 

In Step~3, we have 
\begin{align*}
&\limsup_{\epsilon\rightarrow0}\left|\int_{\Omega}\left(A_{\epsilon}\left(x,\nabla u_{\epsilon}\right),P_{\epsilon}\left(x,M_{\epsilon}\nabla u\right)-P_{\epsilon}\left(x,\Psi\right)\right)dx\right|\\
	&\quad \leq C\sum_{i=1}^{2}\left[\delta^{\frac{p_{1}}{3-p_{1}}}+\delta^{\frac{p_{1}p_{2}}{2p_{2}-p_{1}}}+\delta^{p_{1}}+\delta^{\frac{p_{2}}{p_{2}-1}}\right]^{\frac{1}{p_{i}}},
\end{align*} 		
where $C$ does not depend on $\delta$.

Hence, proceeding as in Step~2, we find that
\begin{align*}
	\displaystyle				
	& \limsup_{\epsilon\rightarrow0}\left|\int_{\Omega}\left(A_{\epsilon}\left(x,\nabla u_{\epsilon}\right),P_{\epsilon}\left(x,M_{\epsilon} \nabla u\right)\right)dx-\int_{\Omega}\left(b(\nabla u),\nabla u\right)dx\right|\\
	&\leq C\left[\sum_{i=1}^{2}\left(\delta^{\frac{p_{1}}{3-p_{1}}}+\delta^{\frac{p_{1}p_{2}}{2p_{2}-p_{1}}}+\delta^{p_{1}}+\delta^{\frac{p_{2}}{p_{2}-1}}\right)^{\frac{1}{p_{i}}}+\delta\left\|b(\nabla u)\right\|_{L^{q_{2}}(\Omega,\mathbb{R}^{n})}\right],
\end{align*}
where $C$ is independent of $\delta$.  Now since $\delta$ is arbitrarily small, the proof of Step~3 is complete.
\bibliographystyle{plain}	
\bibliography{paper}
\end{document}